\documentclass{article}
\usepackage[utf8]{inputenc}
\usepackage[normalem]{ulem}
\usepackage{amsmath,amssymb,amsthm,amsfonts,amstext,amsbsy,amscd,color}
\usepackage{graphicx,float,mathtools}
\usepackage[dvipdfm,a4paper,left=30mm,right=35mm,top=35mm,bottom=35mm]{geometry}

\usepackage{soul}

\newcommand{\X}{\mathcal X}
\newcommand{\C}{\mathcal C}
\newcommand{\f}{\frac}

\newcommand{\R}{\mathbb R}
\newcommand{\p}{\partial}
\newcommand{\ep}{\varepsilon}
\newcommand {\eps}  {\varepsilon}

\DeclareMathOperator{\E}{{\mathbb E}}
\newcommand{\1}{{\mathchoice {\rm 1\mskip-4mu l} {\rm 1\mskip-4mu l}{\rm 1\mskip-4.5mu l} {\rm 1\mskip-5mu l}}}
\newtheorem{remark}{Remark}
\newtheorem{lemma}{Lemma}
\newtheorem{assumption}{Assumption}
\newtheorem{proposition}{Proposition}

\newtheorem{theorem}{Theorem}
\usepackage{cancel}
\newcommand{\beq}{\begin{equation}}
\newcommand{\eeq}{\end{equation}}
\newcommand{\dst}{\displaystyle}
\newcommand{\diff}{\mathrm{d}}

\newcommand{\Dif}{{ D}}

\title{Scaling limits for a population model with growth, division and cross-diffusion}
\author{Marie Doumic\thanks{Inria and CMAP, team MERGE, IP Paris, École polytechnique, CNRS, 
91128 Palaiseau Cédex. Email : marie.doumic@inria.fr} \and Sophie Hecht\thanks{Sorbonne Universit{\'e}, CNRS, Universit\'{e} de Paris, Laboratoire Jacques-Louis Lions UMR7598, F-75005 Paris. 
Email~:~sophie.hecht@sorbonne-universite.fr} \and Marc Hoffmann\thanks{University Paris Dauphine-PSL and Institut Universitaire de France, CEREMADE, Place du Mar\'echal De Lattre de Tassigny, 75016 Paris, France. Email~:hoffmann@ceremade.dauphine.fr} \and Diane Peurichard\thanks{Inria Paris, team MUSCLEES, Sorbonne Universit{\'e}, CNRS, Universit\'{e} de Paris, Laboratoire Jacques-Louis Lions UMR7598, F-75005 Paris. 
Email : diane.a.peurichard@inria.fr}}
\date{\today}

\begin{document}

\maketitle

\begin{abstract}
    Originally motivated by the morphogenesis of bacterial microcolonies, the aim of this article is to explore models through different scales for a spatial population of interacting, growing and dividing particles. We start from a microscopic stochastic model, write the corresponding stochastic differential equation satisfied by the empirical measure, and rigorously derive its mesoscopic (mean-field) limit. Under smoothness and symmetry assumptions for the interaction kernel, we then obtain entropy estimates, which provide us with a localisation limit at the macroscopic level. Finally, we perform a thorough numerical study in order to compare the three modelling scales. 
\end{abstract}



\section{Introduction}

When describing a population of $N$ interacting elements -- such as particles, cells or individuals -- and assuming that the population is large, i.e. $N\to\infty,$ three spatial scales appear. First, at the individual or microscopic scale, a stochastic system of particles may be written, where each individual is characterised by  its characteristic traits, its spatial position and its movement. The interaction between them is regulated by attractive or repulsive forces and/or an external potential. Second, mean-field or kinetic equations correspond to a mesoscopic scale, where the number of particles tends to infinity while the interaction kernel remains nonlocal. Finally, localisation limits can be derived, leading to an aggregation equation / porous medium system, where the interaction range tends to zero. Notably, this third scale may reveal more adequate than the mesoscopic one to describe systems with short-range interactions, where each cell interacts only with a limited number of close neighbours. These two types of limits -- kinetic limits and localisation limits -- have attracted much attention in recent years, we refer to~\cite{golse2016dynamics,chaintron2022propagation,serfaty2024lectures} for recent monographs. In this article, we intend to go one step further by considering a population of cells growing and reproducing by fission, with the case of bacterial microcolony morphogenesis in mind~\cite{doumic2020purely}. The difficulties are twofold: first, growth and division make the system  non-conservative, since both the number of individuals and their total volume or mass change with time; second, the characteristic trait is a continuous variable, unlike the multispecies cases cited above, leading to a lack of compactness.

Let us consider a stochastic number $N_{t}$ of spherical individuals (particles/cells) characterised by their center of mass $(X_i)_{1 \leq i \leq N_t}\in {\R^d}^{N_t}$  and radius $(R_i)_{1 \leq i \leq  N_t} \in [0,R]^{N_t}$. The dynamics of the particle system is described by the following stochastic differential equation system:
\begin{align}
        \diff X_i  & = - \f{\lambda}{N}\sum\limits_{j=1}^{ N_t} \nabla_x K (R_i, R_j,X_i-X_j) \diff t +  \sqrt{2 \Dif} \diff B_{t}^i, \label{eq:transport} \\
        \diff R_i & = g(R_i) \diff t, \label{eq:growth}
\end{align}
where $K$ is the interaction potential, $ B_{t}^i$  are independent Brownian motions,  $\Dif \geq 0$ is the diffusion coefficient and $g(R)$ is the  growth rate of a particle of size $R$. In addition, a particle of size $r$ divides into two daughter cells of size $2^{-1/d}r$ (so that the total size is conserved upon division) with an instantaneous probability rate $\beta(r)$. Upon division, the two daughter cells are positioned according to
$X \pm \alpha 2^{-\f{1}{d}} r \times P(2\pi \theta),$ for some parameter $0 \leq \alpha < 1$, where $\theta \in [0,1]^{d-1}$ and $P(2\pi\theta)$ defines the spherical coordinates, which are  either uniformly randomly chosen in $[0,1]^{d-1}$ or according to a probability law $\kappa(\theta)$\footnote{For $d=2,$ we have $P(2\pi \theta)=(\cos (2\pi \theta),\sin(2\pi\theta))$, and for $d=3$ we have $P(2\pi\theta)=(\sin(2\pi \theta_1)\sin(2\pi\theta_2),\sin(2\pi\theta_1)\cos(2\pi\theta_2),\cos(2\pi\theta_1)).$}. This system is a generalisation of the classical and widely studied system of interacting particles of homogeneous and constant size \cite{oelschlager1990large,PHILIPOWSKI2007526,chaintron2022propagation}. Recently, some works have considered the heterogeneity of a particle population by considering several species interacting through potentials $K_{i,j}$ for $i$ and $j$ two populations \cite{chen2019rigorous}, $1\leq i,j\leq K$. Here we go a step further by considering heterogeneity as given by a continuous trait (here the size of the particle). To our knowledge, this  has not yet been considered in the literature. 

In the microscopic model \eqref{eq:transport}--\eqref{eq:growth} the interaction kernel is scaled by a constant $N$ which represents the order of magnitude of the total number of particles $N_t$ -- for instance, in the numerical study, we assume a deterministic initial number of particles $N_0$ and take $N=N_0$. Heuristically, this gives a total interaction strength of order one when $N$ tends to infinity.  Let us therefore consider the point measure\footnote{For a finite point measure $\mu$ on $\R_+ \times \R^d$ which has then representation $\mu(\diff r, \diff x) = \sum_{i = 1}^n \delta_{(r_i, x_i)}((\diff r, \diff x))$ for some $(r_i, x_i) \in \R_+ \times \R^d$, we use the classical notation $\langle \mu, \varphi\rangle = \sum_{i = 1}^n \varphi(r_i,x_i)$.}
\begin{equation}
\label{eq:mu}\mu_t^N(\diff r,\diff x) = N^{-1}\sum_{i = 1}^{\langle N\mu_t^N, {\bf 1}\rangle}\delta_{(R_i(t),X_i(t))}(\diff r,\diff x),
\end{equation}  
where the sum ranges from $1$ to $ N_t \coloneqq N \langle \mu_t^N, {\bf 1}\rangle$, which is finite if the initial number of cells $N \langle \mu_0^N, {\bf 1}\rangle$ at time $0$ is finite. At the limit $N\rightarrow +\infty$, if $\mu_0^N \to \mu_0^\infty$ in distribution, we prove  that the measure $\mu_t^N$ converges in law toward $\mu_t^\infty$  which is a solution in a weak sense of the equation
\begin{equation}\label{eq:croisfragdiff}
    \left\{\begin{array}{ll}
\f{\p}{\p t} \mu^\infty_t +\f{\p}{\p r}\left(g(r)\mu^\infty_t\right) - \lambda \nabla_x\cdot \left( \mu^\infty_t \nabla_x U_K[\mu^\infty_t]\right) +\beta(r)\mu^\infty_t { - \Dif \Delta_x \mu^\infty_t}
\\ \\  \qquad =\dst\int\limits_{[0,1]^{d-1}} 2^{1+\f{1}{d}}\beta(2^{\f{1}{d}}r) \kappa(\theta) \mu^\infty_t( 2^{\f{1}{d}}r,x \pm \alpha r { P(2\pi\theta)})\diff \theta,
\\ \\
\mu^\infty_{t=0}=\mu^\infty_0,\qquad  g(0)\mu_t^\infty (0,x)=0,
\end{array}
\right.
\end{equation}
where
\begin{equation} 
U_K[\mu](r,x) = \int_{\R_+ \times \R^d} 
{ K( r, r',x-x')}
\mu(\diff r',\diff x').
\end{equation}
The assumptions required and the exact weak convergence result are detailed in Section~\ref{sec:micro-meso} and Theorem~\ref{th: weak limit}. The proof follows the strategy developed in~\cite{fournier2004microscopic,tran2008large} for nonconservative systems of particles. The case of conservative size-homogeneous systems, i.e. when we have neither growth nor division or size structure, has been studied in~\cite{oelschlager1990large,PHILIPOWSKI2007526}. In the case without diffusion, we can also note that the point measure $\mu_t^N$ is an {\it exact} weak solution of~\eqref{eq:croisfragdiff} (with $D=g=\beta=0$), see~\cite[Sec. 1.5]{golse2016dynamics}. 
  
Equation \eqref{eq:croisfragdiff} represents the particle system on a mesoscopic scale. It is a mixture of two well-known equations: the aggregation-diffusion equation \cite{Bertozzi2011,Carrillo2019,di2018nonlinear,Giunta2021} and the growth-fragmentation equation \cite{BP,doumic2023individual}. Since it is not the main focus of this paper, the study of the properties of the new model equation~\eqref{eq:croisfragdiff} is limited to the observations made thanks to numerical simulations. However, it is important to note that this new equation is likely to reproduce interesting phenomena (sorting, blow-up, etc) and therefore it would be interesting to study its long-time behavior  in future work.

In the applications we have in mind, another scaling is important, namely the range of interaction between individuals, which is assumed to be very small compared to the size of the domain, and of the same order of magnitude as the average size of the cells. The aim is then to derive what is  called a {\it macroscopic} or a {\it localisation} limit,  where the interaction kernel converges toward a Dirac delta function in space, so that the interaction becomes local.
To consider this limit, we first write the equation in dimensionless variables, and then consider a scaling such that the limit system is given by
\begin{multline}\label{mactot}
 \partial_t u_0 + \partial_r (g(r) u_0) - 
 \nabla_x\cdot \Big( u_0 (t,r,x) \nabla_x \int_0^R   { \Gamma}(r,r') u_0(t,r',x)\diff r' \Big) - {\Dif} \Delta_x u_0 \\ \\
 = 2^{1+\f{1}{d}}\beta(2^{\f{1}{d}}r) u_0(t, 2^{\f{1}{d}}r,x) - \beta(r) u_0(t,r,x),
\end{multline}
where $u_0$ is the density of particles and $\Gamma(r,s)$ is the integral over space of the interaction potential $K(x,r,s)$ (we have taken $\lambda=1$ for simplicity, see Section~\ref{sec:scalin} for details about the scaling). The full localisation limit has been well studied in the case of equal-sized particles, without growth or fragmentation: the articles already cited~\cite{oelschlager1990large,PHILIPOWSKI2007526} do not carry out only the micro-meso but also the localisation limits, and then the vanishing diffusion limit in~\cite{PHILIPOWSKI2007526}. The article \cite{LMG} also considers the case of a single species without diffusion, and~\cite{carrillo2019blob,carrillo2024nonlocal} developed a gradient-flow approach. These latter studies consider the localisation limits in the context of  developing particle methods for approximating aggregation-diffusion equations. Some recent articles have extended the localisation limit to the case of a finite number of interacting species, see~\cite{ChenDausJungel_2018,JungelPortischZurek_2022} for cases with diffusion, and~\cite{burger2022porous,doumic2024multispecies} without, or still~\cite{david2024degenerate} for the inviscid limit in the case of two species. 

To our knowledge, the question that interests us, namely the localisation limit in  the case of a 'continuous' heterogeneity, modelled by the addition of a size variable, has not been addressed yet. However this limit raises a number of difficult questions. First, growth and division render the system non-conservative, preventing the use of energy estimates. Second, the size of the particles, being a continuous trait, leads to new difficulties in obtaining compactness estimates. For these reasons, we derive  the rigorous localisation limit in the case without growth and division, thus eliminating the first issue and focusing on the second one. The general case is formally derived and illustrated with numerical simulations.

Assuming no growth  ($g=0$) and no fragmentation ($\beta =0$), we denote $n_\ep$ the density of particles, and prove the weak convergence of the sequence $n_\ep$ to $n_0$ solution to the aggregation equation~\eqref{eq:n_0} in Theorem~\ref{TH1}. 
Similar to the method used in various studies \cite{ChenDausJungel_2018,jungel2016entropy,JungelPortischZurek_2022,LMG}, we use entropy dissipations to recover compactness estimates. As is often done~\cite{LMG,doumic2024multispecies,carrillo2019blob,carrillo2024nonlocal}, we make the hypothesis that the kernel function is
an auto-convolution, i.e. there exists $\rho $ such that, with $\check{\rho}(r,x):=\rho(r,-x),$ we have
\beq 
K(r,r',x) = [ \check{\rho}(r,\cdot) *_x \rho(r',\cdot)](x)= \int_{\R^d} \check{\rho}(r,y) \rho(r',x-y)\diff y.
 \label{eq:as3}
\eeq
Without size dependence, this hypothesis can be found throughout the literature \cite{oelschlager1990large,LMG,carrillo2024nonlocal,doumic2024multispecies} since it allows to recover  classical a priori estimates from the entropies $\int n \ln n$ and $\int n \, K\ast_x n$ (these classical entropies are generalised to our system), sometimes called  Shannon-type and Rao-type entropies, respectively~\cite{JungelPortischZurek_2022}. This hypothesis, together with the diffusion term $D>0$, provides us with estimates that  lead to a weak convergence of $n_\ep$. Moreover, considering that $\rho \in  H^1([0,R];L^1(\R^d))$ allows us to recover the compactness for $U_{K_\eps}[n_\eps]$, which allows us to conclude. Note also that both the micro-meso and meso-macro limits consider a smoothness hypothesis on the interaction potential, which allows us to avoid blow-up --  much progress has been made recently to derive limits for non-smooth interaction kernels~\cite{serfaty2024lectures}.

\begin{remark}
For example, the hypothesis~\eqref{eq:as3} on the interaction potential is satisfied by the gaussian kernel
\[
K(r,r',x)= \frac{\gamma(r)\gamma(r')}{(2\pi(r^2+r'^2))^{d/2}} \exp \Big( -\f{|x|^2}{2(r^2+r'^2)} \Big),
\]
with $\gamma \in H^1(\R_+)\cap L^\infty(\R_+)$. We have
\[ \begin{aligned}
K(r,r',x)  = \gamma(r)\gamma(r')  \int_{\R^d} \frac{e^{ -\f{|x-y|^2}{2r^2} }}{(2\pi r^2)^{d/2}}  \frac{e^{ -\f{|y|^2}{2r'^2} }}{(2\pi r'^2)^{d/2}} \diff y = [ \check{\rho}(r,\cdot) *_x \rho(r',\cdot)](x),
\end{aligned}
\]
with $\rho(r,x)= \check{\rho}(r,x)=\frac{\gamma(r)}{(2\pi r^2)^{d/2}} \exp \big( -\f{|x|^2}{2r^2} \big) $. This example is  implemented numerically in Section~\ref{section_numerics}. We could also generalize it to the convolution of two Gaussians with variance defined as a function of $r,$ replacing $r^2$ by some $\sigma^2(r).$
\end{remark}

{


{Finally, since the localisation limit of the general mesoscopic model (with growth and fragmentation) is only formally derived, we provide a thorough numerical study to explore the link between the three modelling scales. The microscopic model is numerically discretized with classical methods (explicit Euler scheme), and the meso- and macro-models are discretized with finite-volume schemes with upwind fluxes \cite{Bailo2020} with special attention to the fragmentation terms. Using appropriate observables, we perform a quantitative comparison between the three models, focusing on the role of the number of particles in the microscopic model. We study three different settings (with growth and without fragmentation, with fragmentation and without growth, and with both), which allow us  to study precisely the role of each phenomenon at the different scales. In all cases, we obtain a good qualitative agreement between the three models, both for the spatial distribution and the size distribution at least in early times, and we  show quantitatively that the micro-meso agreement improves as the number of particles $N$ increases in each case. However, separating the three cases allows us to better understand the role of each phenomenon in the convergence of one scale to the other. Indeed, we first observe that the relative $L^1$-error between the micro and meso densities increases with time in the case with growth alone while it remains constant in the case of fragmentation alone, suggesting that the fragmentation process leads to a longer agreement between the micro and meso models compared to the growth process alone. This highlights the key role of repulsive interactions in the agreement between the micro and meso models. In fact, the mesoscopic model is obtained in the limit of an infinite number of particles, and therefore of interactions. By creating particles, the fragmentation process keeps the number of interactions large, whereas  with the growth process alone, the number of interactions only decreases with time as the particles move further apart due to repulsion. These results are reminiscent of several works in the literature showing that when considering micro to meso limits in interacting particle systems, low density regimes are better captured by a large number of particles \cite{Motsch2017, Degond2022}. 

Another major observation is the good quantitative agreement between the meso and macro models even when the localisation scaling parameter is not small. In all cases, the meso and macro models remain close and the relative $L^1$-error between the two densities even decreases with time. This enables to highlight the fact that we are in a  regime where linear diffusion dominates the non-local effects due to the repulsive interactions. In the appendix, we document the role of nonlinear diffusion by considering smaller linear diffusion coefficients.

Finally, we show that when both phenomena (growth and fragmentation) are combined, the meso and macro models produce mass faster than the microscopic dynamics and as a result, the agreement between the micro and meso and macro models is  observed only for early times of the simulations. We interpret this again by the fact that the mesoscopic and macroscopic models are obtained in the limit of infinite number of particles while the microscopic simulations are done with a finite number of particles, for which errors in the initial condition are amplified by the growth-fragmentation process. Similar observations have been made for systems with short-range repulsion and cell division in~\cite{Motsch2017}.

}
}

The organisation of this paper is schematised in Fig.\ref{fig_schema}. In Section~\ref{sec:micro-meso}, we introduce the stochastic microscopic system (Section~\ref{sec:micro}) and  rigorously derive its mean-field limit in the general case (Section~\ref{sec:meanfield}), including growth, division and interaction (Theorem~\ref{th: weak limit}). In Section~\ref{sec:meso-macro}, Theorem~\ref{TH1}, we perform a dimensional analysis (Section~\ref{sec:dim}) and introduce the scaling that leads to the new mesoscopic model (Section~\ref{sec:scalin}). We then establish the localisation limit in the case  without growth and division (Section~\ref{Sec3.3}) -- this could also be thought of as a local-in-time limit, for systems with slow growth and reproduction and fast interaction. Finally in Section~\ref{section_numerics}, we carry out a thorough numerical study in order to compare the models and their sensitivity with respect to scaling parameters.

\begin{figure}[H]
    \centering
   \includegraphics[scale = 0.22]{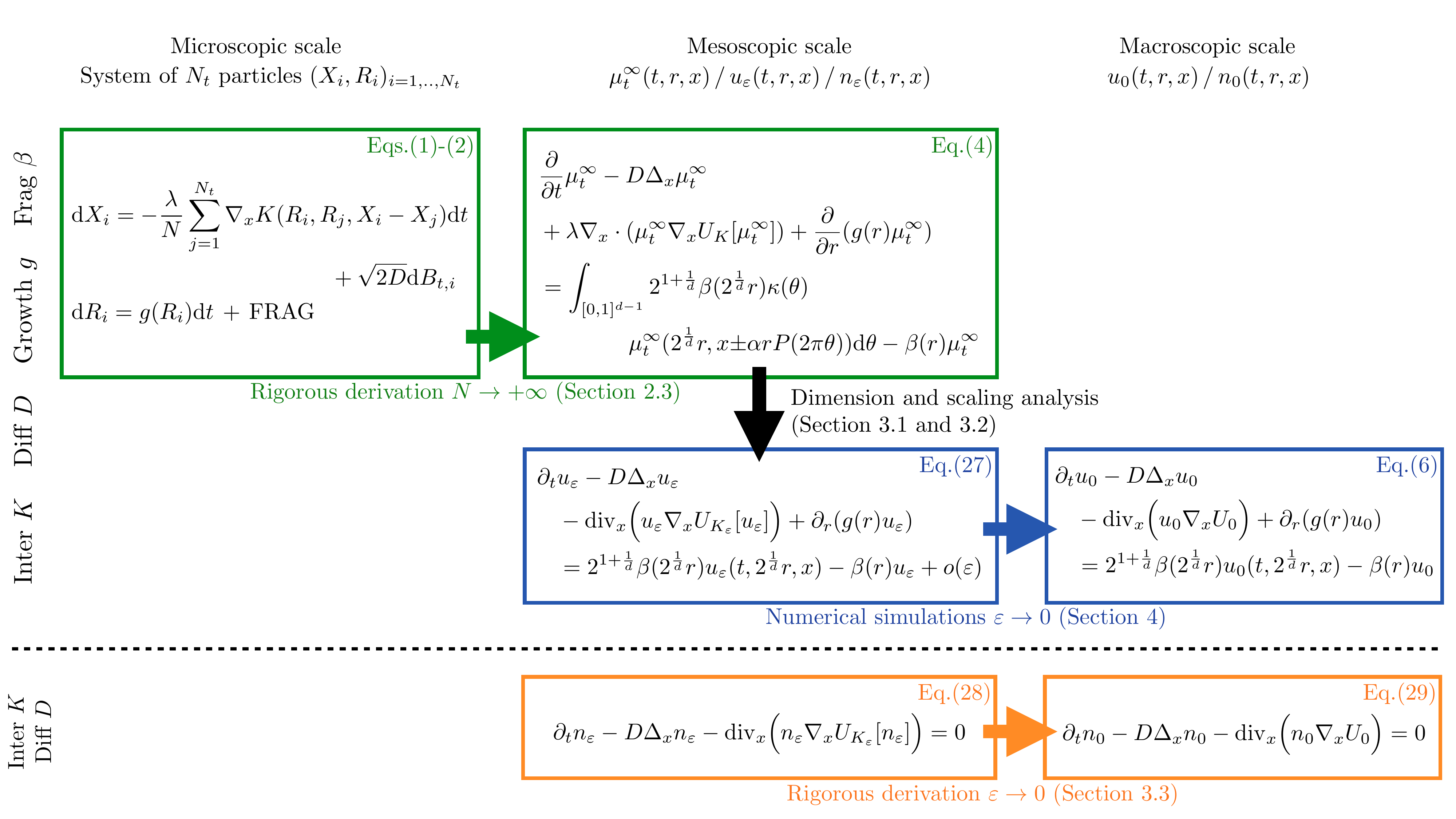}
    \caption{Schematic representation of our main models and results}
    \label{fig_schema}
\end{figure}

\section{A stochastic system of particles with growth, division and interaction}

\label{sec:micro-meso}

{As described in the introduction, we study a large population with a random number of individuals $N_t$, where $N_t$ is of the order of a constant $N \geq 1$. The individuals evolve according  to the spatial interaction and diffusion described by Equation~\eqref{eq:transport}, they grow according to Equation~\eqref{eq:growth}, and divide into two equal-sized daughters, with a rate $\beta(r)$ and a spatial distribution of the daughters  given by the law $\kappa(\theta)$.}
We first give rigorously define the mathematical objects involved in the description of the microscopic model in Section \ref{sec:micro}. In Section \ref{sec: existence micro}, we prove existence and uniqueness of our model as a solution to a non-conservative system of interacting stochastic differential equations with transport, jumps and diffusion. This part is classical and follows \cite{fournier2004microscopic} and \cite{tran2008large}. However, we  need to be careful in some way in order to account for the 
nonlinear (and stochastic) evolution of the system between jumps.

\subsection{Stochastic model} \label{sec:micro}

At time $t$, the size- and space-structured cell population is described by the state (here the size and position) of the living cells, which we denote by
$$\left(\left(r_1(t),x_1(t)\right),\left(r_2(t),x_2(t)\right),\ldots, \left(r_i(t),x_i(t)\right), \ldots \right).$$
We encode this information into the rescaled point measure $\mu_t^N$ introduced in Equation~\eqref{eq:mu}. Abusing the notation slightly, we write $ (r_i(\mu_t^N),x_i(\mu_t^N))=(r_i(t),x_i(t))$.
We denote the state  space of sizes and positions by $\mathcal X =  (0,\infty)\times\R^{ d} $, and the space of finite point measures with values in $\mathcal X$ by $\mathcal M = \mathcal M(\mathcal X)$ . The measure $(N\mu_t^N)_{t \geq 0}$ can be viewed as a random variable taking value in  $\mathcal T \subset \mathbb D([0,\infty), \mathcal M)$, where $ \mathbb D([0,\infty), \mathcal M)$ denotes the set of c\`adl\`ag functions from $[0,\infty)$ with values in the set of non-negative measures on the (closure) of the state space of size and position.  This set $\mathcal T$ is defined as piecewise  continuous finite point measures: $\mu \in \mathcal T$ if $\mu_t$ is a finite point measure and there exists $0 =t_0 < t_1 < \ldots$ with $\lim_n t_n = \infty$ such that $\mu_t $  is continuous for every $t \in [t_i, t_{i+1})$. In particular, this allows us to uniquely define $\mu_{t-}$.\\

We have a complete description of $\mu^N = (\mu_t^N)_{t \geq 0}$ by means of a family of independent Poisson random measures $(M_i(\diff s,\diff \theta, \diff u))_{i \geq 1}$  with intensities $\diff s\otimes \diff \theta \otimes \diff u$ on $\R_+ \times [0,1]^{d-1} \times \R_+$, { and independent Brownian motions $((B_t^i)_{t \geq 0})_{i \geq 1}$, all defined simultaneously} on a sufficiently rich  filtered probability space $(\Omega, \mathcal F, (\mathcal F_t)_{t \geq 0}, \mathbb P)$.\\

It is given by the following stochastic differential equation, written in a weak sense on  test functions $\varphi(t,r,x)$ via the scalar product defined by
\[\langle \mu_t^N, \varphi(t,\cdot, \cdot)\rangle \coloneqq \f{1}{N}\sum_{i = 1}^{N_t}\varphi(t, x_i(\mu_t^N), r_i(\mu_t^N)),\qquad N_t \coloneqq \langle N\mu_t^N, {\bf 1}\rangle,\] as
\begin{equation}\label{eq:stoch:r0}\begin{array}{ll}
& \langle \mu_t^N,\varphi (t,\cdot,\cdot)\rangle   = \langle \mu_0^N,\varphi (0,\cdot,\cdot)\rangle \\ \\
 &+\dst\int\limits_0^t N^{-1}\sum\limits_{i=1}^{\langle N\mu_{s-}^N, {\bf 1}\rangle}{ \int_{[0,1]^{d-1}}}\int_0^\infty\Bigl(2\varphi\big(s,2^{-\f{1}{d}} r_i(\mu_{s-}^N),x_i(\mu_{s-}^N)\pm \alpha 2^{-\f{1}{d}} {r_i(\mu_{s-}^N)} {P(2\pi\theta)}\big) \Bigr.
 \\ \\
 & \Bigl.- \varphi\big(s,r_i(\mu_{s-}^N),x_i(\mu_{s^-}^N)\big)\Bigr) 
\\ \\ &  \times{\1}_{\left\{u \leq \beta\left(r_i(\mu_{s-})\right)\kappa(\theta)\right\}} 
M_i(\diff s,\diff \theta,\diff u) {+ \sqrt{2\Dif} N^{-1}\int_0^t \sum\limits_{i  = 1}^{\langle N\mu_{s-}^N, {\bf 1}\rangle} \nabla_x\varphi(s, r_i(\mu_s^N),x_i(\mu_s^N))\diff B_s^i}\\ \\
& + \int\limits_0^t \big\langle \f{\p}{\p s}\varphi(s,\cdot,\cdot)  + g(r)\f{\p}{\p r}\varphi(s,\cdot,\cdot)  { - \lambda \nabla_x U_K}[\mu_{s^-}^N]\cdot \nabla_x \varphi(s,\cdot,\cdot){ + {  \Dif}\Delta_x \varphi(s,\cdot,\cdot)}, \mu_{s}^N\big\rangle \diff s,
\end{array}
\end{equation}
where
\begin{equation} \label{eq: def noyau}
{U_K}[\mu](r,x) = { \int_{\mathcal X} }
{ K( r, r',x-x')}
\mu(\diff r',\diff x').
\end{equation}
In \eqref{eq:stoch:r0},  we have denoted 
\begin{equation}\label{def:notation1}
2f(y\pm x) \coloneqq f(y+x)+f(y-x)
\end{equation}

\subsection{Existence and uniqueness of \eqref{eq:stoch:r0}} \label{sec: existence micro}

\begin{assumption} \label{hyp: basic} We have
\begin{itemize}
\item (Division rate) $\beta: (0,\infty) \rightarrow \R_+$ is bounded.
\item (Offspring dissemination) $\kappa: [0,1]^{d-1} \rightarrow \R_+$  is a bounded probability density function such that if $P(2\pi \theta_1)=-P(2\pi \theta_2)$ then $\kappa(\theta_1) = \kappa(\theta_2)$ (radial symmetry).
\item (Growth rate) $g: \R_+ \rightarrow \R_+$ is Lipschitz continuous.
\item (Interaction) {$\nabla_x K:  \R_+ \times \R_+\times \R^d $ is continuous and $x \mapsto \nabla_x K(r,r',x)$ is Lipschitz continuous, locally uniformly in $r,r'$.} 
{ \item (Diffusion) Let $D\geq 0$ be the (possibly zero) diffusion coefficient.}
\end{itemize}
\end{assumption}

\begin{theorem} \label{prop: existunicit}
Work under Assumption \ref{hyp: basic}. 
If the $\mathcal F_0$-measurable  finite point  {random} measure $N\mu_0^{N}$ is such that
$\mathbb P(\exists i \neq j, x_i(\mu_0^N) = x_{j}(\mu_0^N))=0$,
and 
$${ \E}\big[\langle N\mu_0^N, {\bf 1}\rangle^p\big] <\infty,\;\;\text{for some}\;\;p\geq 1,$$
then there exists a unique process $(\mu_t^N)_{t \geq 0}$ solution to 
 \eqref{eq:stoch:r0}. Moreover, for every $t>0$, we have 
\begin{equation} \label{eq: moment esti basic}
{ \E}\big[\sup_{0 \leq s \leq t}\langle N\mu_s^N, {\bf 1}\rangle^p\big] <\infty.
\end{equation}
\end{theorem}

\begin{proof}[Proof of Theorem \ref{prop: existunicit}] 
{Let us first show \eqref{eq: moment esti basic}. By \eqref{eq:stoch:r0} and neglecting negative jumps, we have
\begin{align*}
&\sup_{s \leq \min(t,\tau_k)} 
\langle N\mu_s^N, {\bf 1}\rangle  \leq \langle N\mu_0^N, {\bf 1}\rangle+2\int_0^{\min(t,\tau_k)}\sum\limits_{i=1}^{  \langle N\mu_{s-}^N, {\bf 1}\rangle}M_i(\diff s, [0,1]{ ^{d-1}} \times [0,\|\kappa\|_\infty\|\beta\|_\infty]),
\end{align*}
where $\tau_k = \inf\{s \geq 0, \langle N\mu_s^N, {\bf 1}\rangle \geq k\}$ is a localising sequence.
Taking $p$-power and expectation, we obtain 
\begin{align*}
\mathbb E\Big[\sup_{s \leq \min(t,\tau_k)}&\langle N\mu_s^N, {\bf 1}\rangle^p\Big]  \\
&
\leq 2^{p-1}\Big(
\mathbb E[\langle N\mu_0^N, {\bf 1}\rangle^p]+2^p\|\kappa \beta\|_\infty^p\mathbb E\Big[
\big(
\int_0^{\min(t,\tau_k)}\langle 
N\mu_{s^-}^N, {\bf 1}
\rangle \diff s
\big)^p
\Big]
\Big) 
\\
& \leq 2^{p-1}\Big(\mathbb E[\langle N\mu_0^N, {\bf 1}\rangle^p]+2^p\|\kappa\beta\|_\infty^pt^{p-1}\int_0^t \mathbb E\big[\sup_{v \leq \min(s,\tau_k)}\,\langle N\mu_v^N, {\bf 1}\rangle^p\big]\diff s\Big). 
\end{align*}
By Gr\" onwall's lemma, we conclude that for every $t \geq 0$, there exists $C_t >0$ independent of $k$ such that
\begin{equation} \label{eq: gron}
{ \mathbb E}\big[\sup_{s \leq \min(t,\tau_k)}\langle N\mu_s^N, {\bf 1}\rangle^p] \leq C_t.
\end{equation}
As a consequence $\sup_k\tau_k = \infty$  almost surely: by contradiction, assume that for some $t_0 < \infty$, we have $\mathbb P(\sup_k \tau_k \leq t_0) \geq \varepsilon_{t_0} >0$. Then,
from
$$\sup_{t \in [0,\min(t_0, \tau_k)]} \langle N\mu_s^N, {\bf 1}\rangle^p \geq k^p \1_{\{\tau_k \leq  t_0\}}$$
valid for every $k \geq 0$, we infer $\mathbb E[\sup_{s \in [0,\min(t_0, \tau_k)]} \langle N\mu_s^N, {\bf 1}\rangle^p] \geq k^p\varepsilon_{t_0}$ for every $k \geq 0$, which contradicts \eqref{eq: gron} for $t=t_0$. We then let $k \rightarrow \infty$ in \eqref{eq: gron} and obtain \eqref{eq: moment esti basic} by Fatou's lemma.\\

We now prove the existence of the process $(\mu_t^N)_{t \geq 0}$.  Let $T_0=0$. Assume that for some $s \geq 0$, $\langle \mu_s, {\bf 1}\rangle$ is finite (this is at least ensured for $s=0$ since $\langle \mu_0, {\bf 1}\rangle$ is finite almost-surely). Then the jump rate of the population at time $s$ is bounded by 
$$\langle N\mu_s^N, {\bf 1} \rangle \|\kappa \beta\|_\infty.$$ 
It is therefore possible to define almost-surely the increasing sequence of jump times $T_k$ of $N\mu_s^N$ and the process $N\mu_s^N$ can be constructed recursively by a simulation algorithm using acceptance-rejection using the Poisson measures and independent Brownian motions at hand for all times $s < T_\infty = \sup_kT_k$.\\

We only need to show that the evolution of $N\mu_s^N(\diff r,\diff x)$ is well defined on $[T_k, T_{k+1})$ for $k \geq 0$. We need some notation. Since the division mechanism is binary, we have exactly $(k+1)\langle N\mu_0^N, {\bf 1}\rangle  = \langle N\mu_{T_k}^N, {\bf 1}\rangle$ particles alive in the  time interval $[T_k, T_{k+1})$.

Consider now the ordinary differential equation
\begin{equation}\label{eq:growthcharac}\tfrac{d}{dt}R(t,t_0,r)=g(R(t,t_0,r)),\qquad R(t_0,t_0,r)=r,
\end{equation}
for arbitrary $t,t_0, r$.
Since $g$ is Lipschitz continuous by Assumption \ref{hyp: basic}, the map $(t,r) \mapsto R(t,t_0, r)$ is well defined. For $s \in [T_k, T_{k+1})$ and $i = 1,\ldots, \langle N\mu_{T_k}^N, {\bf 1} \rangle = (k+1)\langle N\mu_0^N, {\bf 1}\rangle$,  we then let
$$r_i(\mu_s^N) = 
\left\{
\begin{array}{ll}
R(s, T_k, r_i(\mu_{T_k}^N)) & \text{if}\;b(i) < T_{k}, \\ \\
R(s, T_k, 2^{-\f{1}{d}}r_{i^-}(\mu_{T_k-}^N) ) & \text{if}\;b(i) = T_{k},  
\end{array}
\right.$$
where $b(i)$ denotes the time of birth of the particle $i$. In particular, it corresponds to one of the jump times $T_j$, $j=1,\ldots, k$, and it concerns two newborn particles at each time $T_j$.
Next, conditional on $\mu_{T_k}^N$ and $T_k$, we construct a family of random processes $(X_s^i)_{T_k \leq s < T_{k+1}}$ solution to the system of stochastic differential equations, for $i\leq \langle N \mu_{T_k}^N,1\rangle,$
\begin{equation} \label{eq: eds}
X_{s}^i = x_i(\mu_{T_k}){-} { \lambda}N^{-1}\int_{T_k}^s\sum_{j = 1}^{<N\mu_{T_k}^N,1>} \nabla_x K (r_i(\mu_u^N), r_j(\mu_u),X_u^i-X_u^j)\diff u+{ \sqrt{2\Dif}} (B_{{ s}}^i-B_{T_k}^i),
\end{equation}
for $i=1,\ldots, \langle N\mu_{T_k}^N, {\bf 1}\rangle$ and for $T_k \leq { s} < T_{k+1}$ and independent Brownian motions $(B_t^i)_{t \geq 0}$. The Lipschitz property of $\nabla_x K$ implies that the $X^i$ are well defined and unique. Then, for $s \in [T_k, T_{k+1})$, we let

$$x_i(\mu_s^N) = 
\left\{
\begin{array}{ll}
X_{s}^i & \text{if}\;b(i) < T_{k} \text{ or } s>T_k,  \\ \\
X_s^{i-} \\
+u_i\alpha 2^{-\f{1}{d}}{r_{i^-}(\mu_{T_k-})}P(2\pi\theta) & \text{if}\;b(i) = T_{k} \text{ and } s=T_k,   
\end{array}
\right.$$
where $i-$ denotes the index of the mother cell, and 
where $u_j \in \{\pm 1\}$ with $u_i+u_{i'} = 0$ for the two sister particles $i$ and $i'$ that are such that $b(i) = b(i') = T_k$ and $\theta$ is a random variable with distribution $\kappa$ (and independent of the other stochastic components). We now obtain~\eqref{eq:stoch:r0}. 
}

It remains to show that $T_\infty = \infty$ almost surely. Assume on the contrary that for some $T<\infty$, we have $\mathbb P(T_\infty \leq T) >0$. We cannot have $\{T_\infty \leq T\} \subset \{\lim_k\langle \mu_{T_k},{\bf 1}\rangle = \infty\}$, since this would imply  $\{T_\infty \leq T\} \subset \{\sup_k\tau_k < T\}$ and we have a contradiction with $\sup_k\tau_k = \infty$ almost-surely. Hence there exists $M>0$ and $\mathcal A \subset \{T_\infty \leq T\}$ with $\mathbb P(\mathcal A)>0$ such that for every $k \geq 1$, we have $\langle \mu_{T_k},{\bf 1}\rangle < M$ on $\mathcal A$. Then 
we can always assume that the jump times $T_k$ are obtained as a subsequence of a Poisson counting process with intensity { $M \|\kappa \beta\|_\infty$}, for which $\sup_kT_k=\infty$, a new contradiction. The conclusion follows.

\end{proof}

\subsection{Large population limit}\label{sec:meanfield}

We need to make some more stringent assumptions.
\begin{assumption} \label{hyp: strong} We have:
\begin{itemize}
\item[(i)] (Division rate) $\beta: (0,\infty) \rightarrow \R_+$ is differentiable and $\|\beta'\|_\infty < \infty$.
\item[(ii)] (Offspring dissemination) $\kappa: [0,1]^{ d-1} \rightarrow \R_+$ is a bounded probability density function and { $\kappa$ is symmetrical with respect to the sphere, {\it i.e.} $\kappa(\theta_1+\f{1}{2},\theta_2,\cdots,\theta_{d-1}) = \kappa(\theta)$}.
\item[(iii)] (Growth rate) $g: \R_+ \rightarrow \R_+$ is differentiable.
\item[(iv)] (Interaction)  $\nabla_x K: \R^{d} \times \R_+ \times \R_+ \rightarrow \R_+$ is bounded and satisfies
$$\|\nabla_x K\|_\infty + \|D^2_{xx}K\|_\infty+\|\partial_{r'} \nabla_x K\|_\infty +\|\Delta_x (\nabla_x K) \|_\infty \leq C_K < \infty,$$
where the supremum is taken over $(r,r',x) \in  \R_+ \times \R_+\times \R^{d}$. 
\end{itemize}
\end{assumption}

Let $T > 0$ be a fixed and finite time horizon. 
We let $\mathcal M_F = \mathcal M_F(\mathcal X)$ denote the space of finite measures on $\mathcal X$ of positions and sizes. 

We have the following assumption on the initial condition:

\begin{assumption} \label{hyp: init final}
The initial condition { $\mu_0^N$} is a $\mathcal F_0$-measurable random positive measure taking values in $\mathcal M_F$ that satisfies, for a given constant $r_0>0$,
\begin{itemize}
\item[(v)] $\mathbb P(\exists i \neq j, x_i(\mu_0^N) = x_{j}(\mu_0^N))=0$, $\quad \mathbb P(\forall i, r_i(\mu_0^N) \leq r_0)=1$,
\item[(vi)] $\sup_{N \geq 1}\E\big[\langle { \mu_0^N}, {\bf 1}\rangle^3+\langle {         \mu_0^N},|x|+r\rangle\big] <\infty$, 
\item[(vii)] $\mu_0^N \rightarrow \mu_0^\infty \in \mathcal M_F$ (deterministic limit) { in distribution -- for test functions which are continuous and bounded --} as $N \rightarrow \infty$. 
\end{itemize} 
\end{assumption}

\begin{theorem} \label{th: weak limit}
Work under Assumptions \ref{hyp: strong} and \ref{hyp: init final}. Then for all $T>0$, the sequence $\mu^N= (\mu^N_t)_{0 \leq t \leq T}$ solution of \eqref{eq:stoch:r0} converges in law in $\mathbb D([0,T], \mathcal M_F)$ to a deterministic limit $\mu^\infty = (\mu_t^\infty )_{0 \leq t \leq T} \in \mathcal C([0,T], \mathcal M_F)$ as $N \rightarrow \infty$.\\ 
It is the unique measure-valued function 
that satisfies $\sup_{0 \leq t \leq T}\langle \mu_t^\infty, {\bf 1}\rangle < \infty$ and, using the notation~\eqref{def:notation1}, that is solution of~\eqref{eq:croisfragdiff} in a weak sense, {\it i.e.} it satisfies
{ \begin{align}
\langle \mu^\infty_t,\varphi_t\rangle   & = \langle \mu^\infty_0,\varphi_0 \rangle \nonumber\\ 
& +\int_0^t \int_{\mathcal X \times [0,1]}\Big(2\varphi_s \big(2^{-\f{1}{d}}r,x\pm \alpha 2^{-\f{1}{d}}{r} { P(2\pi\theta)}\big) \nonumber\\ 
&\hphantom{kkkkkk} - \varphi_s\big(r,x\big)\Big)\beta(r)\kappa(\theta) \mu^\infty_s(\diff r,\diff x) 
\diff \theta \diff s \nonumber\\ 
&+ \int_0^t \langle \partial_s \varphi_s+g(r)\partial_r \varphi_s { - \lambda \nabla_x U_K}[\mu^\infty_s]\cdot \nabla_x \varphi_s {+{ \Dif} \Delta_x \varphi_s}, \mu^\infty_s \big\rangle \diff s \label{eq:strongtime}
\end{align}}
for any bounded test functions { $(t,r,x)\in [0,T]\times \mathcal X \mapsto \varphi_t(r,x)$ that are { continuously differentiable with bounded derivative in the time and $r$ variables}} and that are twice continuously differentiable in the $x$ variable with bounded second order derivatives. 
\end{theorem}

{ Note that after integrating by parts~\eqref{eq:strongtime}, $\mu^\infty$ is solution of the equation \eqref{eq:croisfragdiff}
written in a strong sense but to be understood in the weak sense of~\eqref{eq:strongtime}.}

\begin{remark} \label{rem: rzero}
Assumptions \ref{hyp: strong} (iii) and \ref{hyp: init final} (v) imply that for every $i$,  the size of the cell $r_i(\mu_t^N) \leq {R(T,0,r_0) }<\infty$ at all times $0 \leq t \leq T$. Therefore we may assume that { $\mathcal X = (0,R_0] \times \R^d$} { for a given $R_0>0$}. In the proof, we only need this assumption in {\it Step 5}.
\end{remark}

\begin{proof}
We follow the strategy of Theorem 5.3 in \cite{fournier2004microscopic}, see also Section 3 in \cite{tran2008large}.\\

\noindent {\it Step 1. Uniqueness.} 

{ We use a specific form of test functions, namely time-independent test functions  $\psi: \mathcal X \to \R$ continuously differentiable in $r$ and twice continuously differentiable in $x$. For such test functions, ~\eqref{eq:strongtime} becomes} 
\begin{align}
\langle \mu^\infty_t,\psi\rangle   & = \langle \mu^\infty_0,\psi \rangle \nonumber\\ 
& +\int_0^t \int_{\mathcal X \times [0,1]^{d-1}}\Big(2\psi \big(2^{-\f{1}{d}}{r},x{ \pm}\alpha 2^{-\f{1}{d}}{r} P(2\pi\theta)\big) - \psi\big(r,x\big)\Big)\beta(r)\kappa(\theta) \mu^\infty_s(\diff r,\diff x) 
\diff \theta \diff s \nonumber\\ 
&+ \int_0^t \langle g(r)\partial_r \psi { - \lambda \nabla_x U_{K}}[\mu^\infty_s]\cdot \nabla_x \psi {+{ \Dif} \Delta_x \psi}, \mu^\infty_s \big\rangle \diff s \label{eq:strong}.
\end{align} 
Let $\mu, \nu$ be two solutions of \eqref{eq:strongtime} satisfying  
$$\sup_{0 \leq t \leq T}\langle \mu_t+\nu_t, {\bf 1}\rangle < \infty.$$
We equip $\mathcal M_F$ with the $H^{-(1,2)}_\infty$ norm 
$${\|\mu \|_{H_\infty^{-(1,2)}}= \sup_{\|\psi\|_{\mathcal C_\infty^{1,2}} \leq 1}\int_{\mathcal X}\psi(r,x)\mu(\diff r,\diff x),}$$
where we define
\begin{equation} \label{eq: def norm strong}
\|\psi \|_{\mathcal C_\infty^{2,1}} = \|\psi \|_\infty+\|\nabla_x \psi\|_\infty+\|\partial_r \psi\|_\infty+\|\Delta_x \psi\|_\infty.
\end{equation}
Our aim is to obtain an inequality of the type
\begin{align}\label{ineq:Gron}
\|\mu_t-\nu_t\|_{H_\infty^{-{ (1,2)}}}  \leq C(T)  \|\mu_0-\nu_0\|_{H_\infty^{-{(1,2)}}},
\end{align}
which will imply uniqueness. To do so, let first
${ \varphi_t(r,x)}$ be a smooth function on  
$[0,T] \times {\mathcal X}$  such that
$$\sup_{t \in [0,T]}\|{\varphi_t}\|_{\mathcal C^{1,2}_\infty} < \infty.$$ 
We have, for fixed $t\in [0,T]$:
\begin{equation} \label{eq: fund decomp}
\langle \mu_t-\nu_t,{ \varphi_t} \rangle = I + II +III,
\end{equation}
with
\begin{align*}
I & =  \int_0^t \int_{\mathcal X \times [0,1]^{d-1}}\big(2\varphi_s\big(2^{-\f{1}{d}}{r}\big), x{\pm} 2^{-\f{1}{d}} \alpha r P(2\pi\theta) \big) \\
& \qquad - \varphi_s( r,x)\big)   \beta(r)\kappa(\theta) (\mu_s-\nu_s)(\diff r,\diff x) 
\diff \theta \diff s, \\ 
II & =  \int_0^t \langle \tfrac{\partial}{\partial s} \varphi_s+g(r)\tfrac{\partial}{\partial r} \varphi_s   { - \lambda \nabla_x U_K}[\mu_{s}]\cdot \nabla_x \varphi_s+ { \Dif}\Delta_x \varphi_s, \mu_s-\nu_s\rangle \diff s,  \\ 
III&  =  { - \lambda}\int_0^t  \left\langle \left( \nabla_x U_K[\mu_s] - \nabla_x U_K[{ \nu}_s]\right)\cdot \nabla_x \varphi_s, \nu_s \right\rangle \diff s,
\end{align*}
where we used the decomposition
\begin{align*}
& \int_0^t \big(\big\langle { \nabla_x U}_K[\mu_s]\cdot \nabla_x { \varphi_s}, \mu_s\rangle -{ \nabla_x U}_K[\nu_s]\cdot \nabla_x { \varphi_s}, \nu_s\big\rangle \big) \diff s \\
&= \int_0^t \langle { \nabla_x U}_K[\mu_{s}]\cdot \nabla_x { \varphi_s}, \mu_s-\nu_s \rangle \diff s +\int_0^t  \langle ({ \nabla_x U}_K[\mu_s] - { \nabla_x U}_K[\nu_s])\cdot \nabla_x { \varphi_s}, \nu_s \rangle \diff s.
\end{align*}
In order to get rid of the term $II,$ we use the following lemma.
\begin{lemma}
Consider the backward transport-diffusion equation (written in strong form) with terminal condition
\begin{equation} \label{eq: sol free transport term}
\left\{\begin{array}{ll}
\partial_s u_s(r,x) + g(r)\partial_ru_s(r,x)-{ \lambda \nabla_x U}_K[\mu_s] { \cdot} \nabla_x  u_s(r,x) +{ \Dif} \Delta_x u_s(r,x)= 0,\;\;0 \leq s \leq t\\ \\
u_t(r,x)=\psi(r,x),
\end{array}
\right.
\end{equation}
for $\psi\in \mathcal C_\infty^{1,2}$. Under Assumption \ref{hyp: strong}, for every $\mu_s\in {\mathcal C}([0,T],{\mathcal M}_F)$, it admits a classical solution $u$ that satisfies:
$$\sup_{s \in [0,T]}\|u_s\|_{\mathcal C_\infty^{1,2}} \leq C \|\psi\|_{\mathcal C_\infty^{1,2}},$$
where $C$ depends on $g,$ { $({ \nabla_x U}_K[\mu_s])_{0\leq s\leq T}$ and $\Dif.$}
\end{lemma}

{ The proof is immediate, since the change of variable $s'=t-s$ leads to the standard linear drift-diffusion equation, with a pure backward transport in the $r$ variable.} 

Let us now take $\psi \in \C_\infty^{1,2}$ such that $\|
\psi\|_{\mathcal C_\infty^{1,2}} \leq 1.$ We choose $\varphi$ solution to~\eqref{eq: sol free transport term} as the test function in~\eqref{eq:strongtime}, so that we now have $II=0$ in~\eqref{eq: fund decomp}.\\

For the term $I,$ we use that the function
$$H_\varphi(s,x,r) = \int_{[0,1]^{d-1}}\big(2\varphi\big(s,2^{-\f{1}{d}}r, x{ \pm}2^{-\f{1}{d}} \alpha r P(2\pi\theta)\big)- \varphi(s, x,r)\big)\beta(r)\kappa(\theta) \diff \theta$$
satisfies
\begin{align*}
&\sup_{s \in [0,T]}\|H_\varphi(s,\cdot)\|_\infty \leq C \|\varphi\|_\infty\|\beta\|_\infty ,\\
&\sup_{s \in [0,T]}\| \partial_r H_\varphi(s,\cdot)\|_\infty \leq C \left(\|\partial_r \varphi \|_\infty\|\beta\|_\infty+\|\varphi \|_\infty\|\beta'\|_\infty\right),\\
&\sup_{s \in [0,T]}\| \nabla_xH_\varphi(s,\cdot)\|_\infty \leq C \|\nabla_x\varphi\|_\infty\|\beta\|_\infty,\\
&\sup_{s \in [0,T]}\|\Delta_xH_\varphi(s,\cdot)\|_\infty \leq C \|\Delta_x\varphi\|_\infty\|\beta\|_\infty,
\end{align*}
where $C>0$ only depends on $T$ and $d$. Therefore, we have
$$\big| I\big| \leq  C' \int_0^t \|\mu_s-\nu_s\|_{H_\infty^{-(1,2)}}\diff s$$
for some $C' = C''(d,T,\|\beta\|_\infty,\|\beta'\|_\infty)\sup_{{ s} \in [0,t]}\|\varphi({ s},\cdot)\|_{\mathcal C^{1,2}_\infty}\leq C''' \| \psi\|_{\mathcal C^{1,2}_\infty}< \infty$.\\ 
   
To bound the term $III$, we first use  Fubini's theorem,
\begin{align*}
&\langle ({ \nabla_x U_K} [\mu_{s}] - { \nabla_x U_K}[\nu_{s}])\cdot \nabla_x { \varphi_s}, \nu_s \rangle \\
&=-\int_{\X} \left(\int_\X \nabla_x K (r,r',x-x') \left(\mu_s(\diff r',\diff x')- \nu_s(\diff r',\diff x')\right) \right)\cdot \nabla_x \varphi_s (r,x) \nu_s(\diff r,\diff x) 
\\
&= \big\langle -\int_{\mathcal X}\nabla_x K(r,\cdot,x-\cdot)\nabla_x\,\varphi_s(r,x)\nu_s(\diff r,\diff x), \mu_s-\nu_s \big\rangle.
\end{align*}
Moreover, the function 
$$(r',x') \mapsto \widetilde H_\varphi(s,r',x') = -\int_{\mathcal X}\nabla_x K (r,r',x-x')\nabla_x\,\varphi_s(r,x)\nu_s(\diff r,\diff x)$$
satisfies 
\begin{align*}
&\sup_{s \in [0,t]}\|\widetilde H_\varphi(s,\cdot)\|_\infty \leq \|\nabla_x K\|_\infty \|\nabla_x\,\varphi\|_\infty \sup_{s \in [0,T]}\langle \nu_s, {\bf 1}\rangle\\
&\sup_{s \in [0,t]}\|\partial_{r'} \widetilde  H_\varphi(s,\cdot)\|_\infty \leq \|\partial_{r'} \nabla_x K\|_\infty\|\nabla_x\,\varphi\|_\infty \sup_{s \in [0,T]}\langle \nu_s, {\bf 1}\rangle\\
&\sup_{s \in [0,t]}\|\nabla_{x'} \widetilde H_\varphi(s,\cdot)\|_\infty \leq \|D^2_{xx} K\|_\infty \|\nabla_x\,\varphi\|_\infty \sup_{s \in [0,T]}\langle \nu_s, {\bf 1}\rangle,\\
&\sup_{s \in [0,t]}\|\Delta_x \widetilde H_\varphi(s,\cdot)\|_\infty \leq  \|\Delta_x (\nabla_x K)\|_\infty \|\nabla_x \,\varphi\|_\infty \sup_{s \in [0,T]}\langle \nu_s, {\bf 1}\rangle,
\end{align*}

therefore
$$\big|III\big| \leq C \int_0^t\|\mu_s-\nu_s\|_{H_\infty^{-(1,2)}}\diff s$$
with $C  = { C^{st}(\nabla_x K)} \sup_{s \in [0,{ t}]}\|\varphi\|_{\mathcal C^{1,2}_\infty}  \sup_{s \in [0,T]}\langle \nu_s, {\bf 1}\rangle$.
Putting $I$, $II=0$ and $III$ together  we obtain
\begin{align*}
\langle \mu_t-\nu_t,\varphi(t,\cdot) \rangle   =  \langle \mu_t-\nu_t, \psi \rangle   \leq C \int_0^t \|\xi_s-\zeta_s\|_{H_\infty^{-(1,2)}}\diff s.
\end{align*}
where $C$ is a constant depending on $T,\,d,\,K,\,\|\beta\|_\infty, \|\beta'\|_{\infty}$ and $\|\psi \|_{\mathcal C^{1,2}_\infty}.$  Taking the supremum over $|\psi|_{\mathcal C^{1,2}_\infty} \leq 1$ we obtain
\begin{align*}
\|\mu_t-\nu_t\|_{H_\infty^{-{(1,2)}}}  \leq C\int_0^t \|\mu_s-\nu_s\|_{H_\infty^{-(1,2)}}\diff s
\end{align*}
for every $0 \leq t \leq T$. We obtain ~\eqref{ineq:Gron} by Gr\"onwall's lemma, and uniqueness follows.
\\

\noindent {\em Step 2. Moment estimate.} 

We next establish the moment estimate:
\begin{equation} \label{eq: moment esti}
\sup_{N \geq 1} \E^N[\sup_{0 \leq t \leq T}\langle \mu_t^N, {\bf 1}\rangle^3] < \infty.
\end{equation}
Indeed, we can reproduce line by line the beginning of the proof of Proposition \ref{prop: existunicit} with $p=3$ and in the estimate
\begin{align*}
&\sup_{s \leq \min(t,\tau_k)} 
\langle \mu_s^N, {\bf 1}\rangle^p  \\
&\leq 2^{p-1}\big(\langle \mu_0^N, {\bf 1}\rangle^p+\big(2N^{-1}\int_0^{\min(t,\tau_k)}\sum_{i \leq  \langle N\mu_{s-}^N, {\bf 1}\rangle}M_i(\diff s, [0,1]^{d-1} \times [0,\|\kappa\|_\infty\|\beta\|_\infty])\big)^p\big)
\end{align*}
to obtain 
$$
\sup_{N \geq 1}\E\big[\sup_{0 \leq t \leq T}\langle \mu_t^N, {\bf 1}\rangle^3\big] \leq C_T \sup_{N \geq 1}\E\big[\langle \mu_0^N, {\bf 1}\rangle^3\big] 
$$
and we obtain \eqref{eq: moment esti} by Assumption \ref{hyp: init final} {\it (vi)}.\\

\noindent {\em Step 3. Tightness.} 

We now show that the sequence $\mathrm{Law}(\mu^N)$  in $\mathcal P(\mathbb D([0,T], \mathcal M_F))$ of laws of $\mu^N$ is tight, when $\mathcal M_F$ is endowed with the vague topology {(i.e. for test functions that are continuous with compact support)}. The extension to the weak topology { (i.e. for test functions that are continuous and bounded)} is done in {\em Step 6} thanks to a classical argument. Again, by a classical result (see {\it e.g.} around Proposition 2 in \cite{sznitman1991topics}),  the sequence $\mathrm{Law}(\mu^N)$ is tight in $\mathcal P(\mathbb D([0,T], \mathcal M_F))$ if, for every { twice continuously differentiable (bounded) function} $\psi$ on $\mathcal X$, the sequence of the laws of  $\langle \mu^N, \psi\rangle$ is tight in $\mathbb D([0,T], \mathcal X)$.

{ Let us now consider a test function of the form $(s,r,x) \mapsto \varphi_s(r,x)$ which is smooth (differentiable in $s$, in $r$ and twice differentiable in $x$) and that we can consider for fixed $t$ as a function defined on $\mathcal X$. Setting, for $0 \leq s \leq T$,} 

\begin{equation} \label{eq: def Hpsi}
H_\varphi(s,r,x,\theta) = 
2\varphi_s\big(2^{-\f{1}{d}}r_i(\mu_{s-}^N),x_i(\mu_{s-}^N)\pm 2^{-\f{1}{d}}\alpha r_i(\mu_{s-}^N)  P(2\pi\theta)\big) - \varphi_s\big(r_i(\mu_{s-}^N),x_i(\mu_{s-}^N)\big) 
\end{equation}
we have a semimartingale decomposition $\langle \mu_t^N, \varphi_t\rangle = M^N_t(\varphi)+ V^N_t(\varphi)$, where $M_t^N(\varphi) 
$ equals
\begin{align*}
& N^{-1}\int_0^t \sum_{i \leq \langle N\mu_{s-}^N, {\bf 1}\rangle}\int_{[0,1]^{d-1} \times \R_+}H_\varphi(s,r_i(\mu_{s-}^N),x_i(\mu_{s-}^N),\theta) \1_{\{u \leq \beta(r_i(\mu_{s-}^N))\kappa(\theta)\}}M_i(\diff s,\diff \theta, \diff u) \\
&  -\int_0^t \int_{[0,1]^{d-1}}\langle H_\varphi(s,\cdot, \theta)\beta, \mu_s^N \rangle \kappa(\theta)\diff \theta \diff s +\tfrac{\Dif}{N}\int_0^t \sum_{i  = 1}^{\langle N\mu_s^N, {\bf 1}\rangle} \nabla_x\varphi_s( r_i(\mu_s^N),x_i(\mu_s^N))\diff B_s^i.
\end{align*}
It is a squared integrable martingale with predictable bracket
\begin{align*}
\langle M_{\cdot}^N(\varphi) \rangle_t & =  N^{-1}\int_0^t \int_{[0,1]^{d-1}}\langle H_\varphi(s,\cdot, \theta)^2\beta, \mu_s^N \rangle \kappa(\theta)\diff \theta \diff s +{ \Dif} N^{-1}\int_0^t \langle  (\nabla_x\varphi(s,\cdot))^2, \mu_s^N \rangle \diff s,
\end{align*}
and $V^N_t(\varphi) = \langle \mu_t^N, \varphi(t,\cdot)\rangle - M^N_t(\varphi)$ has bounded variation. The tightness of the sequence of the laws of $\mu^N$
 is then a consequence of
\begin{equation} \label{eq: first tight}
\sup_{N \geq 1}\E[\sup_{0 \leq t \leq T}\langle \mu_t^N, \varphi(t,\cdot)\rangle] < \infty,
\end{equation}
\begin{equation} \label{eq: second tight}
\sup_{N \geq 1, S,S'}\E[|M_{S'}^N(\varphi)-M_S^N(\varphi)|] \leq C\delta^{1/2},
\end{equation}
\begin{equation} \label{eq: third tight}
\sup_{N\geq 1, S,S'}\E[|V_{S'}^N(\varphi)-V_S^N(\varphi)|] \leq C\delta,
\end{equation}
for some $C>0$ uniformly over stopping times $0 \leq S\leq S'\leq S+\delta \leq T$
as follows from Aldous' \cite{aldous1978stopping} and Rebolledo's \cite{joffe1986weak} criteria.
Since $\varphi$  is bounded, \eqref{eq: first tight} follows from \eqref{eq: moment esti}. Next
\begin{align*}
\E[|M_{S'}^N(\varphi)-M_S^N(\varphi)|] & \leq \E[\langle M^N(\varphi)\rangle_{S+\delta}-\langle M^N(\varphi)\rangle_{S}]^{1/2} \\
& \leq  N^{-1/2} \big(3\|\varphi \|_\infty \|\beta \kappa \|_\infty^{1/2} +\Dif \|\nabla_x \varphi \|_\infty \big)\E\big[\int_S^{S+\delta} \langle \mu_s^N, {\bf 1}\rangle \diff s\big]^{1/2} \\
\end{align*} 
and \eqref{eq: second tight} follows by \eqref{eq: moment esti}. Finally, $V^N_{S'}-V^N_S$ is bounded by
\begin{align*}
& \int_S^{S+\delta} \Big(\int_{[0,1]^{d-1}}\langle |H_\varphi(\cdot, \theta)|\beta, \mu_s^N \rangle \kappa(\theta)\diff \theta + \langle \partial_s\varphi+g(r)|\partial_r \varphi | { + \lambda|\nabla_x U_{K}}[\mu_s^N]\cdot \nabla_x \varphi | {+ { \Dif}|\Delta_x \varphi|}, \mu_s^N \big\rangle\Big) \diff s \\
&\leq \delta \big(3\|\varphi\|_\infty \|\beta\kappa\|_\infty  +\|\nabla_x K\|_\infty\|\nabla_x \varphi\|_\infty\sup_{0 \leq t \leq T}\langle \mu_t^N, {\bf 1}\rangle + { \Dif} \| \Delta_x \varphi \|_\infty \big)\sup_{0 \leq t \leq T}\langle \mu_t^N, {\bf 1}\rangle
\end{align*}
where we used 
$|{ \nabla_x U_{K}}[\mu_s^N](r,x)| \leq \|\nabla_x K\|_\infty \langle \mu_s^N, {\bf 1}\rangle.$
Thus \eqref{eq: third tight} follows by \eqref{eq: moment esti} likewise.\\

\noindent {\em Step 4.} From
$$\big|\langle \mu_t^N, \varphi_t  \rangle - \langle \mu_{t-}^N, \varphi_t\rangle\big| \leq 3\|\varphi\|_\infty N^{-1},$$
almost-surely, we infer that any process $\mu^\infty$ such that $\mathrm{Law}(\mu^\infty)$ is a limit point of $\mathrm{Law}(\mu^N)$ is almost-surely continuous, in a strong sense,  {\it i.e.} in total variation say.\\

\noindent {\em Step 5.} We now show that almost-surely, any $\mu^\infty$ of {\em Step 4} is the unique solution of~{\eqref{eq:strongtime}}. First,
$\sup_{0 \leq t \leq T}\langle \mu_t^\infty, {\bf 1}\rangle < \infty$ almost surely by \eqref{eq: moment esti}.  For $\nu \in \mathcal C([0,T], \mathcal M_F)$, introduce
\begin{align*}
\Psi_t(\nu) & = \langle \nu_t, \varphi_t\rangle  -  \langle \nu_0,{ \varphi_0} \rangle -\int_0^t \int_{\mathcal X \times [0,1]}H_\varphi(s,x,r,\theta)\beta(r)\kappa(\theta) \nu_s(\diff x,\diff r) 
\diff \theta \diff s \nonumber\\ 
&- \int_0^t \langle \partial_s \varphi+g(r)\partial_r \varphi { - \lambda \nabla_x U_{K}}[\nu_s]\cdot \nabla_x \varphi + { \Dif}\Delta_x \varphi, \nu_s \big\rangle \diff s,\nonumber
\end{align*}
where $H_\varphi$ is defined in \eqref{eq: def Hpsi}. For every $N \geq 1$, we have
\begin{equation} \label{eq: fund mg equality}
M_t^N(\varphi) = \Psi_t(\mu ^N).
\end{equation}
On the one hand,
$$\E[|M_t^N(\varphi)|^2] = \E[\langle M_{\cdot}^N(\varphi)\rangle_t] \leq N^{-1}\big((3\|\varphi\|_\infty)^2\|\beta \kappa\|_\infty  +\Dif^2\|\nabla_x \varphi \|_\infty^2\big)\int_0^t \E[\langle \mu_s^N, {\bf 1}\rangle]\diff s$$
that converges to $0$ thanks to \eqref{eq: moment esti}. On the other hand,  since $\mu^\infty$ is almost surely strongly continuous and $\varphi$ is continuous, using moreover Assumption \ref{hyp: init final}  {\it (vi)} and {\it (vii)}, we have that $\nu \mapsto \Psi_t(\nu)$ is almost-surely continuous\footnote{Here we use the fact that although the mapping $x \in \mathbb D([0,T], \mathcal X) \mapsto x(t)$ is not continuous for the Skorokhod topology, it is continuous at $x$ if $x$ is continuous as a function from $[0,T]$ to $\mathcal X$.} at $\mu^\infty$.
Moreover
$$\Psi_t(\nu) \leq C(1+\sup_{0 \leq s \leq t}\langle \nu_s, {\bf 1}\rangle^2).$$
Therefore, using \eqref{eq: moment esti} again, the sequence $\Psi_t(\mu^N)$ is uniformly integrable and we infer $\E[|\Psi_t(\mu^N)|] \rightarrow \E[|\Psi_t(\mu^\infty)|]$. We conclude $\E[|\Psi_t(\mu^\infty)|]=0$ by \eqref{eq: fund mg equality} and obtain the result.\\

\noindent {\it Step 6.}  Let $\tau >0$ By the preceding results, we have that $\langle \mu^N, \1_{|(r,x)| \leq \tau}\rangle$ converges in law to $\langle \mu^\infty, \1_{|(r,x)| \leq \tau}\rangle$. 
The result holds true replacing $\1_{|(r,x)| \leq \tau}$ by $\bf 1$ up to an error controlled by $\sup_{N \geq 1}\mu^N(\1_{|(r,x)| \geq \tau})+\mu(\1_{|(r,x)| \geq \tau})$. Both terms converge to $0$ as $\tau \rightarrow \infty$. This is a consequence of the fact that $\sup_{N \geq 1}\E\big[\sup_{t \in [0,T]}\langle \mu_t^N,|x|+r\rangle\big] <\infty$, which is proved in the same way as 
 \eqref{eq: moment esti} thanks to the assumption $\sup_{N \geq 1}\E\big[\langle \mu_0^N,|x|+r\rangle\big] <\infty$. (We omit the details.) Hence $\langle \mu^N, {\bf 1}\rangle$ to $\langle \mu^\infty, {\bf 1}\rangle$ in law in $\mathcal D([0,T], \mathcal X)$. Since the limiting process is continuous, the global convergence result also holds true when $\mathcal P(\mathbb D([0,T], \mathcal M_F)$ is equipped with the weak topology, following the criterion proved in Roelly and M\'el\'eard \cite{meleard1993convergences}.\\

The proof of Theorem \ref{th: weak limit} is complete.
\end{proof}

\section{From the mean-field equation to a localisation limit}\label{sec:meso-macro}
%
{ 
 
 \subsection{Dimensionless equations} \label{sec:dim}

We express the problem \eqref{eq:croisfragdiff} in dimensionless variables. From now on, we replace the notation $\mu_t^\infty$ by $u(t,r,x)$ the solution to \eqref{eq:croisfragdiff}, {\it i.e.} we write $u$ for $\mu^\infty$ in the notation of the previous section. Let $t_0$ be the unit of time and $x_0, r_0, \mu_0 = \frac{1}{ r_0 x_0^d}$ be the units of space, size, and distribution function respectively. The scaling of $u$ comes from the fact that it is a probability distribution on $\mathbb{R}_+ \times \mathbb{R}^d $. We define the dimensionless variables:
$$\bar{x} = \frac{x}{x_0}, \bar{r} = \frac{r}{r_0}\; \bar{t} = \frac{t}{t_0}, \; \bar{u}(\bar{t},\bar{r},\bar{x}) = \frac{u(t,r,x)}{\mu_0},$$
and we introduce the following dimensionless parameters:
$$
\bar{\lambda} = \frac{\lambda}{t_0}, \; \bar{g} = g \frac{t_0}{r_0}, \; \bar{\beta} = t_0 \beta, \;  \bar{\Dif} = \Dif \frac{t_0}{x_0^2}, \; \bar{K} { (\bar r,\bar r',\bar x)} = K { (r,r',x)} \frac{t_0^2}{x_0^2}.
$$
We have
$$
\partial_t u(t,r,x) = \frac{1}{t_0 r_0 x_0^d} \partial_{\bar{t}} \bar{u}(\bar{t},\bar{r},\bar{x}),{ \qquad \p_r (gu)=\f{1}{t_0 r_0x_0^d} \p_{\bar r} (\bar g \bar u),}
$$
and { from~\eqref{eq: def noyau}}

\begin{align*}
\nabla_x U_K[u] { (r,x)} &= \int_{\mathbb{R}^d \times \mathbb{R}_+} \nabla_x K(r,r',x-x') u(t,r',x') \diff x'\diff r{ '}\\
&= \int_{\mathbb{R}^d \times \mathbb{R}_+} \frac{x_0^2}{t_0^2} \f{1}{x_0} \nabla_{\bar{x}} \bar{K}(\bar{r},{ \f{r'}{r_0}},\bar{x}-{ \f{x'}{r_0}}) \frac{1}{ r_0 x_0^d} \bar{u}(\bar{t},{ \f{r'}{r_0}},{ \f{x'}{x_0}}) \diff x'  \diff r{ '}\\
&= \frac{x_0}{t_0^2} \int_{\mathbb{R}^d \times \mathbb{R}_+}  \nabla_{\bar{x}} \bar{K}(\bar{r},\bar r',\bar{x}-\bar x') \bar{u}(\bar{t},\bar r',\bar x')  \diff \bar x' \diff \bar r{'}\\
& = \frac{x_0}{t_0^2} \nabla_{\bar{x}} \bar{U}_{\bar K}[\bar{u}]{ (\bar r,\bar x)}.
\end{align*}

In this new set of variables, {thanks to the appropriate links between the scalings}, Equation~\eqref{eq:croisfragdiff} { remains} (omitting the bars for the sake of clarity): 
\begin{multline*}
 \partial_t u + \partial_r (g u) - \lambda \nabla_x \cdot \big(u \nabla_x U_K[u] \big) + \beta(r) u- \Dif \Delta_x u \\ = \int\limits_{[0,1]^{d-1}} 2^{1+\f{1}{d}}\beta(2^{\f{1}{d}}r) \kappa(\theta) u( 2^{\f{1}{d}}r,x {\pm} \alpha r { P(2\pi\theta)})\diff \theta.
\end{multline*}
Finally, choosing the size scale with $R_0$ (see remark \ref{rem: rzero}), the time scale with $\lambda = 1$ and space scale with $\Dif$ gives
\begin{multline*}
 \partial_t u + \partial_r (g u) - \nabla_x \cdot \big( u \nabla_x U_K[u]\big) + \beta(r) u- \Dif \Delta_x u 
 \\
 = \int\limits_{[0,1]^{d-1}} 2^{1+\f{1}{d}}\beta(2^{\f{1}{d}}r) \kappa(\theta) u(t, 2^{\f{1}{d}}r,x {\pm} \alpha r { P(2\pi\theta)})\diff \theta.
\end{multline*}
It is noteworthy that here, we have chosen to follow the system at the diffusion time scale. 

\subsection{Scaling} \label{sec:scalin}
So far, the chosen time and space scales are microscopic ones, and describe the system
at the scale of the agent interactions. In order to describe the system at the macroscopic scale, we introduce a small parameter $\varepsilon \ll 1$ and choose the space, time and size units as $\tilde{x}_0 = \varepsilon^{-1} x_0, \quad \tilde{t}_0 = \varepsilon^{-2} t_0, \tilde{r}_0 = r_0$. The variables $t,r,x$ and unknown $u$ are correspondingly changed to $ \tilde{x} = \varepsilon x, \tilde{t} = \varepsilon^2 t, \tilde{r} = r, u_\varepsilon(\tilde{t},\tilde{r},\tilde{x})= \varepsilon^{-d} u(t,r,x)$. We suppose that $\alpha = O(1)$, that growth and fragmentation are slow processes, i.e $\tilde{g} = \frac{g}{\varepsilon^2}, \tilde{\beta} = \frac{\beta}{\varepsilon^2}$ with $\tilde{g}, \tilde{\beta}$ of order one, and that the interaction function $K$ {acts at the microscopic scale, hence} tends towards a Dirac delta: $$K(r,r',x) = \frac{1}{\varepsilon^d} \tilde{K}(\tilde{r},\tilde{r}',\frac{\tilde{x}}{\varepsilon}),$$
where $\tilde{K}$ is of order 1 { in $L^\infty$}.
We have

    \begin{align*}
  \nabla_x U_K[u](t,r,x) &= \iint_{\mathbb{R}^d \times [0,R]} \nabla_x K(r,r',x-x') u(t,r',x') \diff x'\diff r'\\
  &= \varepsilon \iint_{\mathbb{R}^d \times [0,R]}\frac{1}{\varepsilon^d} \nabla_{\tilde{x}} \tilde{K}(\tilde{r}, \tilde{r'}, \frac{\tilde{x} - \varepsilon x'}{\varepsilon}) \varepsilon^d u_\varepsilon(\tilde{t},\tilde{r}',\varepsilon x') \diff x' \diff \tilde{r}' \\
  & = \varepsilon \iint_{\mathbb{R}^d \times [0,R]}\frac{1}{\varepsilon^d} \nabla_{\tilde{x}} \tilde{K}(\tilde{r}, \tilde{r'}, \frac{\tilde{x} - z}{\varepsilon}) u_\varepsilon(\tilde{t},\tilde{r}',z) \diff z\diff\tilde{r}'\\
  & = \varepsilon \int\limits_0^R [\nabla_{\tilde{x}} K_\varepsilon(\tilde{r},s,.) \ast u_\varepsilon(t,s,.)](\tilde{x}) \diff s,
    \end{align*}

where we have noted $K_\varepsilon(r,r',x) = \frac{1}{\varepsilon^d} \tilde{K}(r,r',\frac{x}{\varepsilon})$. Then, we compute

\begin{multline*}
\int\limits_{[0,1]^{d-1}} 2^{1+1/d} \beta(2^{1/d}) \kappa(\theta) u(t,2^{1/d}r,x \pm \alpha r P(2\pi \theta)) \diff\theta \\ = \int\limits_{[0,1]^{d-1}} 2^{1+1/d} \varepsilon^{2{ +d}} \tilde{\beta}(2^{1/d}) \kappa(\theta) u_\varepsilon(\tilde{t},2^{1/d}\tilde{r},\tilde{x} \pm \varepsilon \alpha r P(2\pi \theta)) \diff \theta \\
=\varepsilon^{2{ +d}} 2^{1+1/d}\tilde{\beta}(2^{1/d}) u_\varepsilon(\tilde{t},2^{1/d}\tilde{r},\tilde{x}) + O(\varepsilon^3),
\end{multline*}
where we have used the fact that $\int\limits_{[0,1]^{d-1}}\kappa(\theta) \diff \theta =1$. Altogether, omitting the tildes for the sake of clarity, we obtain:
\begin{multline}\label{scaledmac}
 \partial_t u_\varepsilon + \partial_r (g u_\varepsilon) - \nabla_x \cdot \bigg(u_\varepsilon \nabla_{x} \int\limits_0^R [ K_\varepsilon({r},s,.) \ast u_\varepsilon(t,s,.)]({x}) \diff s  \bigg) + \beta(r) u_\varepsilon- \Dif \Delta_x u_\varepsilon \\ 
 = 2^{1+1/d} \beta(2^{1/d}) u_\varepsilon(t,2^{1/d}r,x) + O(\varepsilon),    
\end{multline}

\begin{remark}
   Under these scaling assumptions, we account for the fact that particles are very small compared to the space scale. {In this regime, the small parameter $\ep$ takes into account the local character of the interaction, which typically occurs at the same scale as the radius $r$ of a particle, whereas we consider a large number of cells, occupying a domain of much larger size.}  Another possibility would have been to fix the space scale and rescale the size variable. This would be accompanied by different scaling assumptions for the dimensionless parameters. Here, the spatial interaction is kept of order 1 by setting the interaction strength of order $\frac{1}{\varepsilon^d}$. In the limit $\varepsilon \rightarrow 0$, the interaction kernel therefore converges towards a Dirac delta function. This particular regime is usually called the localisation limit.
\end{remark}

{As detailed in the introduction, the rigorous derivation of the localisation limit for this system presents several difficulties { due to (a)} the { growth and division} terms { which lead to a nonconservative equation and prevent the} use of entropy estimates, and { (b) the size-structure of the density which requires new estimates to ensure compactness}.  Therefore in this paper, we present} the rigorous proof of the localisation limit for a simplified system without growth or fragmentation. This work is presented in the next Section \ref{Sec3.3}. The case with growth and fragmentation is illustrated numerically in Section \ref{Sec4}. }

{
\subsection{Localisation limit for a system without growth and fragmentation} \label{Sec3.3}}

Let us consider Equation~\eqref{scaledmac} without the growth and fragmentation terms, namely: 
\beq
\p_t n_\eps (t,r,x) -  \nabla_x\cdot \Big( n_\eps (t,r,x) \nabla_x { U_{K_\eps}[n_\eps]} \Big)= \Dif \Delta_x n_\eps(t,r,x) ,
\label{eq:n}
\eeq
{ with $U_{K_\eps}[n_\eps] = \int_0^R [K_\eps(\cdot, r,s) *_x n_\eps(t,s,\cdot)](x) \diff s $.}
In the following, we establish rigorously the limit $ \eps \rightarrow 0 $ of the model \eqref{eq:n} towards the following equation:
\beq
\p_t n_0 (t,r,x) -  \nabla_x\cdot \Big( n_0 (t,r,x) \nabla_x { U_0[n_0]}\Big)= \Dif \Delta_x n_0(t,r,x),
\label{eq:n_0}
\eeq
{ with $U_0[n_0] = \int_0^R   \Gamma(r,s) n_0(t,s,x)\diff s$ and $\Gamma(r,s) = \int_{\R^d} K (r,s,x) \diff x $}.

\

Before stating our main result, we list a set of regularity assumptions on the interaction function and the initial data.  
{ \begin{assumption}[Interaction function]
   \label{as:interaction} We assume that 
\beq \label{eq:as1}
{\Gamma}(r,s) := \int_{\R^d} K (r,s,x) \diff x  \in L^\infty([0,R]^2),\qquad \forall (r,s) \in [0,R]^2.
\eeq
In addition we suppose that there exists
\beq 
\rho \in  H^1([0,R];L^1(\R^d)) 
\label{eq:as3bis}
\eeq
such that, with $\check{\rho}(r,x):=\rho(r,-x),$ we have
\beq 
K(r,s,x) = [ \check{\rho}(r,\cdot) *_x \rho(s,\cdot)](x).
\eeq
\end{assumption}}
{ The reason for this somewhat technical assumption, used in many and various studies e.g.~\cite{oelschlager1990large,LMG,doumic2024multispecies,carrillo2019blob}, lies in the following} computation. First, we obviously have that $ K_\eps(r,s,x) = [\check{\rho}_\eps(r,\cdot) *_x \rho_\eps(s,\cdot)](x) $. Under Assumption~\ref{as:interaction}, we may write that for a given function $f\in L^p([0,R]\times\R^d)$,
\beq \begin{aligned} \label{sys:calcul_convol}
  \int_{ \R^d}  \int_0^R  \int_0^R &  f(r,x) \; [K_\eps(r,s,\cdot) *_x f(s,\cdot)](x)  \; \diff s \diff r  \diff x    \\
  & =  \int_{ \R^d} \int_0^R  \int_0^R  f(r,x) \; [\check{\rho}_\eps(r,\cdot) *_x \rho_\eps(s,\cdot) *_x f(s,\cdot)](x)  \; \diff s \diff r \diff x
\\
&= \int_{ \R^d}  \int_0^R  \int_0^R  [\rho_\eps(r,\cdot)  *_x f(r,\cdot)](y) \; [   \rho_\eps(s,\cdot)  *_x f(s,\cdot)](y)  \; \diff s \diff r \diff y \\
 & =  \int_{ \R^d} \Big( \int_0^R [\rho_\eps(r,\cdot)  *_x f(r,\cdot)](y)  \; \diff r \Big)^2 \diff y \geq 0
\end{aligned}
\eeq
This type of computation will be used several times in the following section, and is a key point for several estimates.

{ \begin{assumption}[Initial data]\label{as:init}
    We define ${ \mathcal X_R}=[0,R]\times \R^d$ and we assume that 
\beq 
\begin{cases}
n^{ini}_{\varepsilon} := n_\eps(0,\cdot,\cdot)   \geq 0, \qquad  n^{ini}_{\varepsilon} \in L^\infty([0,R];L^1(\R^d)\cap L^2(\R^d)), \\
 \dst\int_{{ \mathcal X_R}}|x|^2 n^{ini}_{\varepsilon} \diff r \diff x <+\infty, \qquad n^{ini}_{\varepsilon} \ln(n^{ini}_{\varepsilon}) \in L^1({ \mathcal X_R}),
\end{cases}
 \label{eq:as4}
\eeq
{ with uniform bounds with respect to $\eps$ in the respective functional spaces.}
\end{assumption}}
We are now in position to state the main result of this section.
\begin{theorem} \label{TH1}
Let $T>0$ and $R>0$. We define $  \mathcal X_{T,R}=[0,T] \times \mathcal X_R $. 
Let $K_\eps$ and $(n^{ini}_{\varepsilon})$ satisfy Assumptions~\ref{as:interaction} and~\ref{as:init}. 
After extraction of subsequences, the density $n_{\varepsilon}$ solution to~\eqref{eq:n} converges weakly in $L^{1+ \f 2 d}({ \mathcal X_{T,R}})$ as $\varepsilon \rightarrow 0$ to a limit $n_0 \in L^{1+\f 2 d}({ \mathcal X_{T,R}})$ which satisfies~\eqref{eq:n_0} { in the sense of distributions}.
\end{theorem}

The main difficulty of the proof lies in the nonlinearity of the term $\nabla \cdot (n_\eps \nabla U_{K_\eps}[n_\eps] )$ in \eqref{eq:n}, which requires strong compactness of $U_{K_\eps}[n_\eps]$ to pass to the limit. In Section~\ref{sec:a priori}, we show different a priori estimates which we use in Section~\ref{sec:compactness} to show the compactness of $(U_{K_\eps}[n_\eps])$ in  $L_{loc}^1([0,T]\times { [0,R] \times } \R^d)$, thus allowing us to extract a strongly converging subsequence. Finally in Section~\ref{sec:convergence}, we prove that the limit of an extracted subsequence of $n_\eps$ is solution of~\eqref{eq:n_0}, which concludes the proof of Theorem~\ref{TH1}. 

{ Let us also note that, having already proved by Theorem~\ref{th: weak limit} the existence and uniqueness of a weak solution to~\eqref{eq:n}, the series of a priori estimates used for Theorem~\ref{TH1} implies, under Assumptions~\ref{as:interaction} and~\ref{as:init}, the existence of a strong solution.}

\subsubsection{A priori estimates}\label{sec:a priori}

In this section we present estimates for $n_\eps$ and $ q_\eps$ defined as follows:
\beq\label{def:qeps} q_\eps(t,x) \coloneqq \int_0^R  [\rho_\eps(r,\cdot)  *_x n_\eps(t,r,\cdot)](x) \diff r .\eeq
Then we can write
\[ \begin{aligned}
 { U_{K_\eps}[n_\eps]} (t,r,x)&=\int_0^R [K_\eps(\cdot, r,r') *_x n_\eps(t,r',\cdot)](x) \diff r'  = \int_0^R  [ \check{\rho}_\eps(r,\cdot) *_x \rho_\eps(r',\cdot) *_x n_\eps(t,r',\cdot)](x) \diff r' \\
& = \int_{\R^d} { \check{\rho}}_\eps(r,x-y) \Big( \int_0^R   [\rho_\eps(r',\cdot) *_x n_\eps(t,r',\cdot)](y) \diff r' \Big) \diff y \\
& = [ \check{\rho}_\eps(r,\cdot)  *_x q_\eps(t,\cdot)](x).
\end{aligned}
\]
Thus Equation~\eqref{eq:n} can be rewritten as
\beq
\p_t n_\eps(t,r,x)  -  \nabla_x\cdot \Big( n_\eps  \nabla_x [\check{\rho}_\eps(r,\cdot)  *_x q_\eps(t,\cdot)](x) \Big)= \Dif \Delta_x n_\eps(t,r,x) , \qquad \mbox{ on } \;{ \mathcal X_{T,R}}.
\label{eq:n2}
\eeq

\begin{proposition}
{ Work under Assumptions~\ref{as:interaction} and~\ref{as:init} and let }$n_\eps$ be solution to~\eqref{eq:n2}. Then 
$$n_\eps \geq 0 \quad \mbox{ on } \quad { \mathcal X_{T,R}}.$$
In addition, for $t\in[0,T]$,  uniformly  { with respect to} $\ep$,
$$  \int_{ { \mathcal X_R}} |x|^2 n_\eps(t,r,x) \diff r \diff x   < +\infty \quad \mbox{ and } \quad \int_{ { \mathcal X_R}} n_\eps(t,r,x) \,|\ln(n_\eps(t,r,x))| \diff r \diff x  < +\infty . $$
Moreover we have the following estimates on $n_\eps$  uniformly  { with respect to} $\ep$,\
\begin{align}
& n_\eps \in L^\infty([0,T]\times[0,R];L^1(\R^d)) \label{eq:nL1} \\
& \sqrt{n_\eps} \in L^2([0,T]\times[0,R];H^1(\R^d)) \label{eq:sqrtnH1} \\
& n_\eps \in L^{1+ \f 2d}({ \mathcal X_{T,R}})  \label{eq:nLp} \\
&  \nabla_x n_\eps \in L^{\f {d+2}{d+1}}({ \mathcal X_{T,R}}),  \label{eq:DnLp} 
\end{align}
and the following estimates on $q_\eps$  defined by~\eqref{def:qeps},  uniformly { with respect to} $\ep$,
\begin{align}
& q_\eps \in L^\infty([0,T];L^1(\R^d)) \label{eq:qL1} \\
& q_\eps \in L^2([0,T];H^1(\R^d)) \label{eq:qH1} .
\end{align}
Finally we also have  uniformly  { with respect to} $\ep$
\beq
n_\eps |\nabla_x  \check{\rho}_\eps *_x q_\eps |^2 \in L^{ 1 }({ \mathcal X_{T,R}}) .\label{eq:est1}
\eeq

\end{proposition}

\begin{proof}
\noindent {\em Step 1. Positivity and $L^1$ bounds~\eqref{eq:nL1} and~\eqref{eq:qL1}.} 

Let us first prove that $n_\eps \geq 0$. We multiply the equation \eqref{eq:n2} by $-\1_{n_\ep { \leq} 0}$ { and use the notation $|n_\ep |_- = - n_\ep \1_{n_\ep \leq 0}$} to get
\[\p_t |n_\ep |_- \leq \nabla_x\cdot \Big ( |n_\ep |_- \nabla_x [\check{\rho}_\eps  *_x q_\eps] \Big) +   \Dif \Delta_x |n_\eps |_- .
\]
The last term is found using the so-called Kato's inequality 
$\Delta f(y) = f''(y) |\nabla_x y|^2 +f'(y) \Delta y \geq f'(y) \Delta y $ for $f(y) = { |y|_-}$ and $y = n_\ep$, { see~\cite{brezis2004kato}}.
Integrating in space gives
\[ \f{\diff}{\diff t}  \int_{\R^d} |n_\ep |_- \leq 0.
\]
Hence $|n_\ep(t,r,x)|_- \leq |n^{ini}_{\varepsilon}(r,x)|_-=0$, for all $(t,r,x)\in { \mathcal X_{T,R}}$, which gives us  positivity. Additionally, integrating~\eqref{eq:n2} in space gives
\[ \frac{\diff}{\diff t} \int_{\R^d} n_\ep (t,r,x) \diff x= \int_{\R^d} \nabla_x\cdot \Big( n_\eps (t,r,x) \nabla_x [\check{\rho}_\eps(r,\cdot)  *_x q_\eps(t,\cdot)](x) \Big) +  \int_{\R^d} \Dif \Delta_x n_\eps(t,r,x)  = 0 . \]
Therefore $ \int_{\R^d} n_\ep (t,r,x) \diff x= \int_{\R^d} n^{ini}_{\varepsilon}(r,x) \diff x \leq \| n_\ep^{ini} \|_{L^\infty\left([0,R];L^1(\R^d)\right)}$ and  we have estimate \eqref{eq:nL1}. Moreover
\[ \begin{aligned}
\| q_\eps \| _{|L^1(\R^d)} &  \leq \int_0^R \| \rho_\eps(r,\cdot)  *_x n_\eps(t,r,\cdot) \| _{|L^1(\R^d)}  \diff r \\
& \leq \int_0^R \| \rho_\eps(r,\cdot) \| _{|L^1(\R^d)} \| n_\eps(t,r,\cdot) \| _{|L^1(\R^d)} \diff r \\
& \leq \| \rho \| _{|L^1([0, R]\times\R^d)} \| n_\eps(t,\cdot,\cdot) \| _{|L^\infty([0,R];L^1(\R^d))},
\end{aligned} \]
so we have estimates \eqref{eq:qL1} uniformly with respect to $\ep$. 

\medskip
\noindent {\em Step 2. $L^2$ bounds~\eqref{eq:qH1} and~\eqref{eq:est1} -- { "Rao-type" entropy inequality}.}
For the  $L^2$ control, we consider the equation satisfied by $q_\eps,$ { obtained by integrating the convolution product of~\eqref{eq:n2} with $\rho_\ep$ on $[0,R]$:} 
\beq \label{eq:q}
\p_t q_\eps (t,\cdot)= \int_0^R \rho_\eps(r,\cdot) *_x \p_t n_\eps (t,r,\cdot) \diff r =  \nabla_x\cdot \Big(  \int_0^R \rho_\eps(r,\cdot) *_x [ n_\eps(t,r,\cdot) \nabla_x \check{\rho}_\eps (r,\cdot) *_x q_\eps ] \diff r\Big ) +  \Dif \Delta_x q_\eps.
\eeq
We multiply equation \eqref{eq:q} by $q_\eps$ and integrate over space
\[ \begin{aligned}
 \f{\diff }{\diff t} \int_{ \R^d} |q_\eps|^2    &=  -  \int_{ \R^d} \int_0^R (\rho_\eps *_x [ n_\eps \nabla_x \check{\rho}_\eps *_x q_\eps ])\diff r \cdot\nabla_x q_\eps -  \Dif \int_{ \R^d} | \nabla_x q_\eps|^2 \\
 &= - \int_{ \R^d}  \int_0^R [ n_\eps \nabla_x  \check{\rho}_\eps *_x q_\eps ]\cdot ( \check{\rho}_\eps *_x \nabla_x q_\eps) \diff r - \Dif \int_{ \R^d} | \nabla_x q_\eps|^2 \\
 &= - \iint_{{ \mathcal X_R}}    n_\eps |\nabla_x \check{\rho}_\eps *_x q_\eps |^2 -  \Dif \int_{ \R^d} | \nabla_x q_\eps|^2 .
 \end{aligned}
\]
Thus by integrating over time
\[ \int_{ \R^d} |q_\eps|^2 + \Dif \int_0^T \int_{ \R^d} | \nabla_x q_\eps|^2  +   \int_0^T \int_{{ \mathcal X_R}}  n_\eps |\nabla_x  \check{\rho}_\eps *_x q_\eps |^2  \leq \int_{ \R^d} |q^0_\eps|^2.
\]
Given that $\int_{ \R^d} |q_\eps(t,\cdot)|^2 \leq R \| \rho_\eps \| _{|{H^1([0,R];L^1(\R^d)}}^{{ 2}} \| n^0_\eps \| _{|L^\infty([0,R];L^2(\R^d))}^{{ 2}} \in L^\infty([0,T])$, it proves the results \eqref{eq:qH1} and \eqref{eq:est1}. 

\begin{remark} This last computation of the $L^2$ norm of $q_\eps$ is equivalent to considering the classical entropy $\int_{\R^d} \int_0^R n_\eps (r,x)  \int_0^R [K_\eps(\cdot, r,s) *_x n_\eps(s,\cdot)](x) \diff s \diff r \diff x${, sometimes called a "Rao-type entropy".}
\end{remark}

\medskip
\noindent {\em Step 3. Second moment control.} 

{ Multiplying~\eqref{eq:n2} by $\vert x\vert^2$ and integrating on $\mathcal X_R$,} we have
\[ \begin{aligned}
 \f{\diff}{\diff t}  \int_{{ \mathcal X_R}} |x|^2 n_\eps  & =  \int_{{ \mathcal X_R}} |x|^2 \nabla_x\cdot [n_\ep  \nabla_x (  \check{\rho}_\eps *  q_\ep)] +  \Dif \int_{{ \mathcal X_R}} |x|^2 \Delta_x n_\ep\\
 & = - \int_{{ \mathcal X_R}} 2 n_\ep x \cdot \nabla_x ( \check{\rho}_\eps *  q_\ep) -\Dif \int_{{ \mathcal X_R}} 2x \cdot \nabla_x n_\ep  \\
 & \leq  2 \Big(\int_{{ \mathcal X_R}} |x|^2 n_\ep \Big)^{1/2} \Big( \int_{{ \mathcal X_R}} n_\ep |\nabla_x   \check{\rho}_\eps * q_\ep|^2\Big)^{1/2} +  2 { d} \Dif \int_{{ \mathcal X_R}} n_\ep 
\end{aligned}
\]
 Hence after integration, denoting ${ V}_\eps(t) \coloneqq \Big(   \int_{{ \mathcal X_R}} |x|^2 n_\eps \Big)^{1/2}$, we have
\[
 { V}^2_\eps(t) \leq { V}^2_\eps(0) +  2\Dif { d} T  \| n_\ep \|_{L^\infty([0,T];L^1( { \mathcal X_R}))} + 2 \int_0^t { V}_\eps(s) \Big( \int_{{ \mathcal X_R}} n_\ep |\nabla_x  \check{\rho}_\eps * q_\ep|^2\Big)^{1/2}(s) \diff s  ,
\]
and thanks to~\cite[Theorem 5]{dragomir},
\[
{ V}_\eps(t) \leq { V}_\eps(0) +   \sqrt{2\Dif  { d} T \| n_\ep \|_{L^\infty([0,T];L^1( { \mathcal X_R}))}} +   \int_0^T \Big( \int_{{ \mathcal X_R}} n_\ep |\nabla_x  \check{\rho}_\eps * q_\ep|^2\Big)^{1/2} . 
\]
Applying the Cauchy-Schwarz inequality to the second member of the right-hand side of the equation and taking the square, we obtain 
\[  \int_{{ \mathcal X_R}} |x|^2 n_\eps \leq   3\int_{{ \mathcal X_R}} |x|^2 n^{ini}_{\varepsilon} +  6\Dif  { d} T  \| n_\ep \|_{L^\infty([0,T];L^1({ \mathcal X_R}))} +  3 T \int_{\mathcal X_{T,R}} n_\ep |\nabla_x  \check{\rho}_\eps * q_\ep|^2  < +\infty
\]
thanks to~the initial condition \eqref{eq:as4}, to~\eqref{eq:nL1} and to~\eqref{eq:est1}.

\medskip
{\em Step 4. $H^1$ norm estimate~\eqref{eq:sqrtnH1} for $\sqrt{n_\eps}$ and entropy estimate.}  We consider the classical { "Shannon-type"} entropy $ \mathcal{E}(n_\eps) \coloneqq \int_{{ \mathcal X_R}} n_\eps \ln n_\eps$ and compute its derivative

\[\begin{aligned}
\f{\diff}{\diff t} \int_{\mathcal X_R} n_\eps \ln n_\eps &= \int_{\mathcal X_R} (1+ \ln n_\eps) \p_t n_\eps
\\
& = - \int_0^R \int_{\R^d} \nabla_x n_\eps \cdot \nabla_x [\check{\rho}_\eps *_x q_\eps] - \Dif \int_{\mathcal X_R}\f{1}{n_\eps} |\nabla_x n_\eps|^2 \\
& = -  \int_{\R^d} | \nabla_x q_\eps |^2 - 4  \Dif \int_{\mathcal X_R}|\nabla_x \sqrt{n_\eps}|^2.
\end{aligned} \]
Then after integrating in time 
\[   \int_{{ \mathcal X_R}} n_\eps |\ln n_\eps| +   \int_0^T  \int_{\R^d}| \nabla_x q_\eps |^2 +  4 \Dif   \int_0^T  \int_{{ \mathcal X_R}} | \nabla_x \sqrt{n_\eps} |^2  \leq \int_{{ \mathcal X_R}} n^0_\eps \ln n^0_\eps + \int_{{ \mathcal X_R}} n_\eps |\ln n_\eps|_-.
\]
Moreover $\int_{{ \mathcal X_R}} n_\eps |\ln n_\eps|_-  \diff r \diff x$ can be decomposed as follows:
\[\int_{{ \mathcal X_R}} n_\eps |\ln n_\eps|_- \diff x \diff r= \int_{{ \mathcal X_R}} n_\eps |\ln n_\eps|_- \1_{n_\ep \geq e^{-|x|^2}} \diff x \diff r + \int_{{ \mathcal X_R}} n_\eps |\ln n_\eps|_- \1_{n_\ep \leq e^{-|x|^2}} \diff x \diff r.
\]
We then bound each term, noticing first that if $n_\ep \geq e^{-|x|^2}$ then $ |\ln n_\eps|_- =-\ln (n_\eps) \1_{n_\ep \leq 1} \leq |x|^2$
\[
\int_{{ \mathcal X_R}}n_\eps |\ln n_\eps|_- \1_{n_\ep \geq e^{-|x|^2}} \diff x \diff r =\int_{{ \mathcal X_R}}n_\eps |\ln n_\eps|_- \1_{n_\ep \geq e^{-|x|^2}}  \diff x
\leq 
\int_{{ \mathcal X_R}} |x|^2 n_\eps \diff x \diff r < +\infty,
 \]
 and, noticing now that if $x\geq 1$ and if $y \leq e^{-|x|^2}$ then $ y |\ln (y)| \leq e^{-|x|^2} |\ln (e^{-|x|^2})|$ because $u\mapsto u|\log(u)|$ is increasing on $(0,e^{-1})$ (with a maximum on $e^{-1}$) so that
 \begin{multline*}
   \int_{  { \mathcal X_R}} n_\eps |\ln n_\eps|_- \1_{n_\ep \leq e^{-|x|^2}}\diff x \diff r 
  \\ \leq  \int_0^R\int_{|x|\leq 1}n_\eps |\ln n_\eps|_- \1_{n_\ep \leq 1} \diff x \diff r
  +
    \int_0^R \int_{ |x|\geq 1} |x|^2 e^{-|x|^2} \diff x \diff r < + \infty.  
 \end{multline*} 
 This gives  $ \nabla_x \sqrt{n_\eps} $ in $ L^2({ \mathcal X_{T,R}})$. Together with estimates \eqref{eq:nL1} we have estimate \eqref{eq:sqrtnH1}. {Notice that we have used the assumption $D>0$ in a crucial way here for \eqref{eq:sqrtnH1}, but that estimate \eqref{eq:qH1} would still hold with $D=0$.}

\medskip
{\em Step 5. $L^{1+\f 2 d}$ norm estimate~\eqref{eq:nLp} for $n_\eps$ and $L^{\f {d+2}{d+1}} $ norm estimate~\eqref{eq:DnLp} for $\nabla_x n_\eps$. }

{ We want to apply the Gagliardo-Nirenberg interpolation inequality on $\R^d$, which states that for $1\leq q, r\leq \infty$ and $\theta\in [0,1]$ such that $\f{1}{p}=\theta(\f{1}{r}-\f{1}{d}) + \f{1-\theta}{q}$ there exists some constant $C$ depending only on $d,q,r,\theta$ such that for any $f:\R^d\to \R$ we have
\[\Vert f \Vert_{L^p(\R^d)}  \leq C \Vert \nabla f \Vert_{L^r(\R^d)}^\theta \Vert  f \Vert_{L^q(\R^d)}^{1-\theta}.\]
We apply it to $f=\sqrt{n_\ep}$, $q=r=2$ and}
 $\theta =\frac{d}{d+2},$ { so that $p=2+\f{4}{d}$,} we integrate on $[0,T]\times[0,R],$ and we have
\[\begin{aligned}
\| n_\eps \|^{1+\f 2 d} _{L^{1+\f 2 d}({ \mathcal X_{T,R}})} &= \int_0^T \int_0^R \| \sqrt{n_\eps} \|^{2+\f 4 d} _{L^{2+\f 4 d}(\R^d)} \\
& \leq C \int_0^T \int_0^R \| \nabla_x \sqrt{n_\eps} \|^{(2+\f 4 d)\theta} _{L^2(\R^d)} \: \| \sqrt{n_\eps} \|^{(2+\f 4 d)(1-\theta)}  _{L^2(\R^d)} \\
& \leq C \int_0^T \int_0^R \: \| \nabla_x \sqrt{n_\eps} \|_{L^2(\R^d)}^{ 2} \: \| \sqrt{n_\eps}  \|^{\f 4 d} _{L^2(\R^d)} \\
& \leq C \: \| \nabla_x \sqrt{n_\eps} \|_{L^2({ \mathcal X_{T,R}})}^{ 2} \: \| n_\eps \|^{\f 2 d}  _{L^{\infty}([0,T]\times[0,R];L^1(\R^d))},
\end{aligned} \]
which { thanks to~\eqref{eq:nL1} and~\eqref{eq:sqrtnH1},} gives \eqref{eq:nLp}. Then given H\"older's inequality \footnote{ applied under the form $\Vert fg\Vert_{L^r}\leq \Vert f\Vert_{L^p} \Vert g\Vert_{L^q}$ if $\f{1}{r}=\f{1}{p}+\f{1}{q}$} with $ \frac{d+1}{d+2} = \frac{1}{2} +\frac{d}{2d+4}$ we have
\[\begin{aligned}
\| \nabla_x n_\eps \|_{L^{\frac{d+2}{d+1}}({ \mathcal X_{T,R}})} & =  2\| \sqrt{n_\eps}  \nabla_x \sqrt{n_\ep} \|_{L^{\frac{d+2}{d+1}}({ \mathcal X_{T,R}})} \\
& \leq 2\| \sqrt{n_\eps} \|_{L^{\frac{2d+4}{d}}({ \mathcal X_{T,R}})} \|  \nabla_x \sqrt{n_\ep} \|_{L^{2}({ \mathcal X_{T,R}})} \\
& \leq 2\| n_\eps \|^{ 2}_{L^{1+\f 2 d}({ \mathcal X_{T,R}})} \| \nabla_x \sqrt{n_\ep}\|_{L^{2}({ \mathcal X_{T,R}})}.
\end{aligned} \]
This provides us with the last inequality~{ \eqref{eq:DnLp}} and concludes the proof.
\end{proof}

\subsubsection{Compactness}\label{sec:compactness}

{The estimates~\eqref{eq:nLp} and~\eqref{eq:DnLp} provide us with sufficient compactness for $n_\ep,$ so that
we now focus on the variable $ { U_{K_\eps}}$}, for which we recall that it can be rewritten as
        $$ { U_{K_\eps}[n_\eps]}(t,r,x) =  \int_0^R [K_\eps(\cdot, r,r') *_x n_\eps(t,r',\cdot)](x) \diff r' =  \check{\rho}_\eps(r,\cdot) *_x q_\eps(t, \cdot) (x)  .$$ 
Given estimate \eqref{eq:qL1}  we have 
\[
     \| {U_{K_\eps}[n_\eps]} (t,\cdot,\cdot)\|_{L^1({ \mathcal X_R})}
\leq \| \check{\rho_\ep}\|_{L^1({ \mathcal X_R})} \| q_\ep (t,\cdot)\|_{L^1(\R^d)} 
\\ \leq  \| \rho\|_{L^1({ \mathcal X_R})} \| q_\ep \|_{L^\infty\left(0,T;L^1(\R^d)\right)} \leq C,
\]
and given estimate~ \eqref{eq:qH1} 
\[\begin{array}{lll}
 \int_{{ \mathcal X_{T,R}}} |{ U_{K_\eps}[n_\eps]} (t,r,x)|^2 \diff r \diff x \diff t
&\leq &\int_0^T \int_0^R \| \check{\rho}_\ep (r,\cdot)\|_{L^1(\R^d)}^2 \| q_\ep (t,\cdot) \|_{L^2(\R^d)}^2 \diff r \diff t
\\
&\leq &\| \rho \|_{L^2\left(0,R;L^1(\R^d)\right)}^2
\|q_\ep \|_{L^2([0,T]\times \R^d)}^2 \leq C
\end{array}
\]
and
\[\begin{array}{lll}
 \int_{{ \mathcal X_{T,R}}} |\nabla_x { U_{K_\eps}[n_\eps]} (t,r,x)^2 \diff x\diff r \diff t
&\leq &\int_0^T \int_0^R \| \check{\rho}_\ep (r,\cdot)\|_{L^1(\R^d)}^2 \| \nabla_x q_\ep (t,\cdot) \|_{L^2(\R^d)}^2 \diff r \diff t
\\
&\leq &\| \rho \|_{L^2\left(0,R;L^1(\R^d)\right)}^2
\|q_\ep \|_{L^2([0,T]; H^1(\R^d))}^2 \leq C
\end{array}
\]
hence uniformly  { with respect to} $\ep$
$$ 
{ U_{K_\eps}[n_\eps]} \in L^\infty([0,T],L^1([0,R]\times\R^d)) \qquad \mbox{ and } \qquad  { U_{K_\eps}[n_\eps]} \in L^2([0,T]\times[0,R],H^1(\R^d)) .
$$ 
In addition we have $\p_r { U_{K_\eps}[n_\eps]} = (\p_r \check{\rho}_\eps) *_x q_\eps $, thus thanks to assumption \eqref{eq:as3bis} 
 \[\begin{array}{lll}
 \int_{{ \mathcal X_{T,R}}} |\partial_r { U_{K_\eps}[n_\eps]} (t,r,x)|^2 \diff r\diff x \diff t
&\leq &\int_0^T \int_0^R \| \partial_r\check{\rho}_\ep (r,\cdot)\|_{L^1(\R^d)}^2 \| q_\ep (t,\cdot) \|_{L^2(\R^d)}^2 \diff r \diff t
\\
&\leq &\| \partial_r \rho \|_{L^2\left(0,R;L^1(\R^d)\right)}^2
\|q_\ep \|_{L^2({ \mathcal X_{T,R}})}^2 \leq C,
\end{array}
\] so finally
\beq \label{eq:pH1}
{ U_{K_\eps}[n_\eps]} \in L^2([0,T],H^1([0,R] \times \R^d)) .
\eeq 
Therefore we have immediately
\beq \label{es:WKF}
\begin{array}{cc} 
    &  { U_{K_\eps}[n_\eps]} \in L^1_{\mbox{\scriptsize{loc}}}({ \mathcal X_{T,R}}),  \\
    &  \nabla_x { U_{K_\eps}[n_\eps]} \in L^1_{\mbox{\scriptsize{loc}}}({ \mathcal X_{T,R}}),  \\
    &  \p_r { U_{K_\eps}[n_\eps]} \in L^1_{\mbox{\scriptsize{loc}}}({ \mathcal X_{T,R}}) , 
\end{array}
\eeq
and { writing the equation satisfied by $U_{K_\ep},$ obtained by integrating in $[0,R]$ the convolution product of~\eqref{eq:n2} with $K_\ep,$ we have}
\beq
\p_t { U_{K_\eps}[n_\eps]} = \nabla_x\cdot \Big( \underbrace{\int_0^R K_\eps(\cdot, r,r') *_x \Big[ n_\eps(t,r',\cdot)  \nabla_x  { U_{K_\eps}[n_\eps]}(t,r',\cdot) + \Dif \nabla_x { n_\eps(t,r',\cdot)}  \Big] (x) \diff r' }_{P_\ep(t,r,x)}\Big).
\eeq
We now prove that $P_\ep(t,r,x) \in L^1({ \mathcal X_{T,R}})$  uniformly  { with respect to} $\ep$, since:
\begin{align*}
    \| P_\ep\|_{L^1({ \mathcal X_{T,R}})} & \leq \int_0^T \int_0^R \int_{\R^d} \int_0^R  |K_\eps(\cdot, r,r') | *_x \Big| n_\eps(t,r',\cdot)  \nabla_x { U_{K_\eps}[n_\eps]}(t,r',\cdot)  + \Dif \nabla_x { n_\eps(t,r',\cdot)}  \Big|(x) \diff r' \diff x \diff r \diff t  
    \\
    &  \leq \int_0^T \int_0^R \int_0^R \| K_\eps(\cdot , r,r')\|_{L^1(\R^d)} \Big[ \| n_\eps(t,r',\cdot)  \nabla_x { U_{K_\eps}[n_\eps]}(t,r',\cdot) \|_{L^1(\R^d)} 
    \\ &\qquad + \Dif  \| \nabla_x { n_\eps(t,r',\cdot)} \|_{L^1(\R^d)} \Big] \diff r'\diff r \diff t
    \\
    &  \leq R \| K(\cdot , r,r')\|_{L^\infty([0,R]^2,L^1(\R^d))} \Big[  \| n_\eps \|_{L^\infty([0,T]\times[0,R],L^1(\R^d))}^{\f 1 2}  \| n_\eps |\nabla_x {\check{\rho}}_\eps *_x q_\eps |^2 \|_{L^1({ \mathcal X_{T,R}})}^{\f 1 2} 
    \\
     & \qquad + { 2 D} \| \sqrt{ { n_\eps(t,r',x)}} \|_{L^2({ \mathcal X_{T,R}})} \| \nabla_x \sqrt{ { n_\eps(t,r',x)}} \|_{L^2({ \mathcal X_{T,R}})} \Big],
\end{align*}
{ where we have done two Cauchy-Schwarz inequalities, writing respectively 
\[n_\ep \nabla_x U_{K_\ep} =n_\ep^{\f{1}{2}} n_\ep^{\f{1}{2}} \nabla_x (\check{\rho}_\ep * q_\ep),\qquad \text{and} \qquad\nabla_x n_\ep = 2\sqrt{n_\ep} \nabla_x \sqrt{n_\ep}.\]
Finally, the right-hand side is uniformly bounded with respect to $\ep$
thanks to the assumption~\eqref{eq:as1} and to the estimates~\eqref{eq:nL1},~\eqref{eq:est1} and~\eqref{eq:sqrtnH1}.}

Then, with the same  method as for the multispecies case~\cite[Section 3.2.]{doumic2024multispecies}, we can show that for a compact set $K \subset { [0,R]\times}\R^d$ and for any $k>0$, we have
\begin{multline} \label{eq:Comp_t}
\Vert { U_{K_\eps}[n_\eps]}(t+k,r,x) - { U_{K_\eps}[n_\eps]}(t,r,x) \Vert_{L^1((0,T)\times K)} 
\\ \leq C(K,T) k \Vert { U_{K_\eps}[n_\eps]} \Vert_{L^2((0,T), H^1( K))} \leq C(K,T){ \sqrt{k}}.
 \end{multline}
 
 Then, given \eqref{es:WKF}, ${ U_{K_\eps}[n_\eps]}$ satisfies the assumptions of the Weil-Kolmogorov-Frechet theorem on $L^1((0,T)\times K),$ hence the sequence ${ U_{K_\eps}[n_\eps]}$ is compact in this space. 

\subsubsection{Convergence}\label{sec:convergence}

Given the previous estimates, we can extract a subsequence (we do not change the notation for the sake of clarity) such that
\begin{align} 
 { U_{K_\eps}[n_\eps]} & \rightarrow U_0 \quad \mbox{ strongly in } \quad L^1((0,T)\times(0,R);L^1_{\mbox{\scriptsize{loc}}}(\R^d)),
 \label{sys:lim_L1strong}\\ 
 n_\ep  & \rightharpoonup  n_0 \quad \mbox{ weakly in } \quad L^{1+\f 2 d}({ \mathcal X_{T,R}}), \label{sys:lim_nL1weak} \\
\nabla_x n_\ep  & \rightharpoonup  \nabla_x  n_0 \quad \mbox{ weakly in } \quad L^{\f{d+2}{d+1}}({ \mathcal X_{T,R}}) \label{sys:lim_nH1weak}.
\end{align}
 In addition, from the uniform bound  ${ U_{K_\eps}[n_\eps]} \in L^2({ \mathcal X_{T,R}})$ thanks to~\eqref{eq:pH1}, and the strong limit \eqref{sys:lim_L1strong}, we have
 \beq \label{sys:lim_Lpstrong}
 {U_{K_\eps}[n_\eps]}   \rightarrow U_0 \quad \mbox{ strongly in } \quad L^p((0,T){ \times (0,R)};L^
q_{\mbox{\scriptsize{loc}}}(\R^d)) \quad \forall 1\leq p < 2, \; 1\leq q < 2.
 \eeq

 We are left to show the convergence in equation \eqref{eq:n}. Let us define $\phi \in C_c^\infty({ \mathcal X_{T,R}})$. Then multiplying equation \eqref{eq:n} by $\phi$ and integrating by parts, we have 
\begin{align*}
 \int_{\mathcal X_{T,R}}  n_\eps \p_t \phi &=  \int_{\mathcal X_{T,R}}  n_\eps \nabla_x { U_{K_\eps}[n_\eps]} \cdot \nabla_x \phi + \Dif \int_{\mathcal X_{T,R}}   \nabla_x n_\eps \cdot  \nabla_x \phi \\ 
 & =  - \int_{\mathcal X_{T,R}} \left(\nabla_x n_\eps \cdot \nabla_x \phi \right)  {U_{K_\eps}[n_\eps]}  - \int_{\mathcal X_{T,R}}  n_\eps \left(\Delta_x  \phi  \right) { U_{K_\eps}[n_\eps]}  + \Dif \int_{\mathcal X_{T,R}}  \nabla_x n_\eps \cdot  \nabla_x \phi.
 \end{align*}
 The left-hand side converges thanks to the weak convergence \eqref{sys:lim_nL1weak} of $n_\eps$. The first and second terms of the right-hand side converge thanks to the strong convergence \eqref{sys:lim_Lpstrong}  of ${ U_{K_\eps}[n_\eps]}$ and the weak convergence \eqref{sys:lim_nH1weak} of $\nabla_x n_\eps$ and \eqref{sys:lim_nL1weak} of $ n_\eps$  respectively. Finally, the last term converges thanks to the weak convergence \eqref{sys:lim_nH1weak} of $\nabla_x n_\eps$. Then passing to the limit and integrating by part, we have
\beq \label{eq:eqlim1}
\int_0^T \int_0^R \int_{\R^d}  n_0 \p_t \phi =  \int_0^T \int_0^R \int_{\R^d}  n_0 \nabla_x U_0 \cdot \nabla_x \phi + \Dif \int_0^T \int_0^R \int_{\R^d}   \nabla_x n_0 \cdot  \nabla_x \phi .
\eeq

Finally we are left to show that $U_0 = \int_0^R {{ \Gamma}(r,s) }n_0{ (t,s,x)\diff s}$. Let us define $\phi \in C_c^\infty({ \mathcal X_{T,R}})$, then
$$
\begin{aligned}
     & \Big| \int_{\mathcal X_{T,R}} \Big ({ U_{K_\eps}[n_\eps]}(t,r,x) - \int_0^R  { \Gamma}(r,s) n_0(t,s,x) \diff s \Big ) \phi(t,r,x) \diff x \diff r \diff t \Big|\\
     & \qquad \leq  \Big| \int_{\mathcal X_{T,R}}\Big ( \int_0^R  [K_\eps(\cdot, r,s) *_x n_\eps(t,s,\cdot)](x) - { \Gamma}(r,s) n_0(t,s,x)  \diff s \Big ) \phi(t,r,x) \diff x \diff r \diff t \Big| \\
     & \qquad \leq  \Big|\int_0^T \int_0^R \int_0^R \int_{\R^d}  \Big( [K_\eps(\cdot, r,s) *_x n_\eps(t,s,\cdot)](x) -  { \Gamma}(r,s)n_\eps(t,s,x) \Big) \phi(t,r,x) \diff x \diff r \diff s \diff t \Big| \\
      & \qquad \qquad + \Big| \int_0^T \int_0^R \int_0^R \int_{\R^d}  \Big( n_\eps(t,s,x) - n_0(t,s,x) \Big) { \Gamma}(r,s)\phi(t,r,x) \diff x \diff r \diff s \diff t \Big|
      \\
      & \qquad \leq  \underbrace{ \int_0^T \int_0^R \int_0^R \int_{\R^d}  \Big| [K_\eps(\cdot, r,s) *_x\phi(t,\cdot,r)](x) - { \Gamma}(r,s)\phi(t,r,x) \Big| n_\eps(t,x,s) \diff x \diff r \diff s \diff t }_{I_\eps}  \\
      & \qquad \qquad + \underbrace{ \int_0^T \int_0^R \int_0^R \int_{\R^d}  \Big| n_\eps(t,s,x) - n_0(t,s,x) \Big||{ \Gamma}(r,s)| | \phi(t,r,x)| \diff x \diff r \diff s \diff t }_{{II}_\eps}
\end{aligned}
$$
For the first term, applying H\"older's inequality with $1 =  \frac{2}{2+d}+\frac{d}{2+d}$ we have
  $$
  I_\eps  \leq \| [K_\eps(\cdot, r,s) *_x\phi(t,\cdot,r)](x) -  { \Gamma}(r,s) \phi(t,r,x) \|_{L^{\frac{2+d}{{ 2}}}([0,T]\times[0,R]^2\times\R^d)} \, R \| n_\eps\|_{L^{\frac{2+d}{{ d}}}({ \mathcal X_{T,R}})}.
$$
Since $\phi \in C_c^\infty({ \mathcal X_{T,R}})$ and $\int_{\R^d} K_\eps(x, r,s) \diff x={ { \Gamma}(r,s)} $, we have for all $(r,s)\in[0,R]^2$, thanks to the continuity in $r$ and $s$ of ${ \Gamma}$ linked to Assumption~\eqref{eq:as3bis}
  $$ [K_\eps(\cdot, r,s) *_x\phi(t,r,\cdot)] \rightarrow {{ \Gamma}(r,s)} \phi(t,r,\cdot)  , \qquad \mbox{ in } L^\frac{2+d}{{ 2}}(\R^d).$$
In addition 
$$  \|[K_\eps(\cdot, r,s) *_x\phi(t,\cdot,r)](x) - {{ \Gamma}(r,s)}\phi(t,r,x)\|_{L^{\frac{2+d}{d}}(\R^d)} \leq { 2 \|{ \Gamma}\|_{L^\infty([0,R]^2)}} \|\phi(t,\cdot,r)\|_{L^{\frac{2+d}{2}}(\R^d)}  \in L^1([0,T]\times[0,R]^2.)  $$
Thus thanks to the dominated convergence theorem,
 $$  \| [K_\eps(\cdot, r,s) *_x\phi(t,r,\cdot)](x) - { { \Gamma}(r,s)} \phi(t,r,x) \|_{L^{\frac{2+d}{2}}([0,T]\times[0,R]^2\times\R^d)} \rightarrow 0 ,$$
 and since $ n_\eps \in L^{1+{ \f 2d}}({ \mathcal X_{T,R}})$ uniformly  { with respect to} $\eps$ by~\eqref{eq:nLp} we have $  I_\eps\rightarrow 0 $.
 
The second term can be rewritten as
  $$ {II}_\eps  \leq \|{ \Gamma}\|_{L^\infty([0,R]^2)}  \int_0^T \int_0^R \int_{\R^d}  \Big| n_\eps(t,s,x) - n_0(t,s,x) \Big|   \Big( \int_0^R| \phi(t,r,x)| \diff r \Big) \diff x \diff s \diff t ,
  $$
  with $ \int_0^R| \phi(t,r,x)| \diff r \in L^{1+\f d2}({ \mathcal X_{T,R}})$, therefore thanks to the weak convergence \eqref{sys:lim_nL1weak} we have $  {II}_\eps\rightarrow 0 $ and $p_0  (t,r,x)= \int_0^R  { \Gamma}(r,s) n_0 (t,x,s)\diff s$. Thus given \eqref{eq:eqlim1}, $n_0$ is solution of \eqref{eq:n_0}, { and this ends the proof of Theorem~\ref{TH1}}.


\

 We can now deduce formally the localisation limit of the complete model \eqref{scaledmac}. Denoting by $u_0(t,r,x)$ the formal limit of $u_\varepsilon(t,r,x)$ solution of \eqref{scaledmac}, we  deduce that $u_0$ is solution of the local equation~\eqref{mactot}.

\section{Numerical simulations} \label{Sec4}

\label{section_numerics}
In this section, we numerically investigate the link between the microscopic model of Section~\ref{sec:micro-meso}, the mesoscopic model given by Equation~\eqref{eq:croisfragdiff} and the macroscopic model given by Equation~\eqref{mactot}. For all simulations, we consider a spatial square domain $\Omega \subset \mathbb{R}^2$ taken large enough so that the discrete particles do not go outside the domain in the observation timeframe, and with no-flux boundary conditions for the mesoscopic and macroscopic models.

\subsection{Numerical setting for the microscopic model}
 We throw initially $N_0$ particles  uniformly distributed in a ball centered in the center of the domain and having radius $S$. The $N_0$ individual particle radii are initially chosen randomly from a uniform distribution $\mathcal{U}([r_{\min}, r_{\max}])$. We consider discrete times $t^n = \sum_{i=0}^n \Delta t^i$ with some carefully chosen time step $\Delta t^n$ (see below), and we consider a splitting scheme for the different mechanisms (spatial motion, growth and division). 

\textit{Space motion of particles.}
The equation of motion for the particles given by Equation~\eqref{eq:transport} is solved using an explicit Euler scheme: Given a configuration at time $t^n$ $(X^n_i, R^n_i)_{1\leq i \leq N_n}$, the displacement of the particles during a time step $\Delta t^n$ is given by
$$
X_i^{n+1/2} = X_i^n - \frac{\Delta t^n}{N} \sum_{j=1}^{N_n} \nabla K \big(X_i^n-X_j^n, R_i^n,R_j^n\big) + \sqrt{2 \Dif \Delta t^n} \; \eta_i,
$$
where $\eta_i$ is randomly chosen from a normal distribution $\mathcal{N}(0,1)$. For all simulations, we consider the following interaction potential:
\begin{equation}\label{choiceK}
K(x,r,s)= \frac{\gamma(r)\gamma(s)}{(2\pi(r^2+s^2))^{d/2}} \exp \Big( -\frac{|x|^2}{2(r^2+s^2)} \Big),
\end{equation}
with $\gamma(r) = r$. 

\textit{Particle growth.}
The particle radii are actualized between times $t^n$ and $t^n + \Delta t^n$ according to 
$$
R_i^{n+1/2} = \min(r_{\max}, R_i^n + \Delta t^n g(R_i^n)).
$$
In the following, we consider a constant growth term $g(r) = g$. Notice that the growth law has been truncated to ensure that particle radii do not exceed the threshold $r_{\max}$.

\textit{Particle division.}
The division process is modelled by a Poisson process of radial-dependent frequency $\beta(r)$: for each cell $i$, the probability to divide between time steps $t^n$ and $t^n + \Delta t^n$ is given by
$$
\mathbb{P}{(\text{cell }}i{\text{ divides between }}t^n{\text{ and }}t^n + \Delta t^n) = 1 - e^{- \beta(R_i^{n+1/2}) \Delta t^n}.                                     
$$
We use a rejection method to decide upon the division of a cell based on this probability. To ensure that cell radii do not exceed $r_{\max}$, we virtually consider that $\beta(r_{\max}) = + \infty$, i.e cells of maximal radius divide with probability one. 

When a cell $i$ divides, its radius and position are actualized according to
$$\begin{cases}
X_i^{n+1} \leftarrow X_i^{n+1/2} - \alpha \frac{R_i^{n+1/2}}{\sqrt{2}} (\cos(\theta_i), \sin(\theta_i))\\
R_i^{n+1} \leftarrow \frac{R_i^{n+1/2}}{\sqrt{2}},
\end{cases}
$$
where $\theta_i$ is chosen randomly from a uniform distribution $\mathcal{U}([0,2\pi]$. Simultaneously, a new cell of radius $\frac{R_i^{n+1/2}}{\sqrt{2}}$ is created at position $X_i^{n+1/2} + \alpha \frac{R_i^{n+1/2}}{\sqrt{2}} (\cos(\theta_i), \sin(\theta_i))$. 

The positions and radii of all the cells $k$ that were not subjected to division in between $t^n$ and $t^n + \Delta t$ are then set to 
$$
\begin{cases}
X_k^{n+1} \leftarrow X_k^{n+1/2} \\
R_k^{n+1} \leftarrow R_k^{n+1/2}.
\end{cases}
$$

We choose the fragmentation rate function $\beta(r)$ such that
$$
\beta(r) = \begin{cases}
0 \qquad \qquad \qquad \text{   for $r<\sqrt{2}\; r_{\min}$}\\
\bar{\beta} \frac{r-\sqrt{2}\; r_{\min}}{r_{\max} - \sqrt{2}\; r_{\min}} \; \text{   for $r\in [\sqrt{2}\; r_{\min}, r_{\max})$}\\
+ \infty \qquad \qquad \quad \text{for $r\geq r_{\max}$}
\end{cases}
$$
Finally, we ensure that the motion of particles during two time steps does not exceed a given threshold by using an adaptative time step. More specifically, we set
$$
\Delta t^n = \min\big(\frac{\delta}{{ \frac{1}{N_0}}\max_i |\sum_{j=1}^{N(t)} \nabla { K} (|X_i^n-X_j^n|, R_i^n,R_j^n|)|}, \frac{\delta^2}{4D}, \frac{0.1r_{\min}}{||\beta||_{L^1}}, \frac{0.1r_{\min}}{2g}  \big).
$$

The numerical parameters used for the microscopic simulations are summarized in Table~\ref{tab:micro_param}. 

\begin{table}
    \centering
    \begin{tabular}{|c|c|c|}
         \hline
    Parameters & Value & Description\\
    \hline
         $x_{\min}$ & -50 & Minimal domain boundary in x- and y- directions\\
         $x_{\max}$ & 50 & Maximal domain boundary in x- and y- directions\\
         $r_{\min}$ & 0.2 & Minimal cell radius\\
         $r_{\max}$ & 1 & Maximal cell radius\\
         $N_0$ &  $[500, 5000]$& Initial number of cells\\
         $\Dif$ & $0.01$ & Diffusion coefficient\\
         $g$ & 0.008 & Growth rate\\
         $\bar{\beta}$ & 0.05 & Maximal division rate (for $r<r_{\max}$)\\
         $\alpha$ & 0.1 & Position of the daughter cells after division\\
         $T_f$ & 100 & Simulation time\\
         $S$ & 2 & Radial support of the initial condition\\
         \hline
    \end{tabular}
    \caption{Model parameters for the microscopic simulations}
    \label{tab:micro_param}
\end{table}

\subsection{Numerical settings for the meso and macro models} For the mesoscopic and macroscopic models, the spatial domain $\Omega$ is discretized into $N_x \times N_x$ regularly spaced points with space step $\Delta x$ and the radial domain $[r_{\min}, r_{\max}]$ is discretized into $N_r$ regularly spaced points with radial step $\Delta r$.

We consider discretes times $t_n = n \Delta t$ for $n\geq 0$ and introduce a Cartesian 3D mesh consisting of the cells $C_{ijk} = [x_{i-1/2}, x_{i+1/2}] \times [y_{j-1/2}, y_{j+1/2}] \times [r_{k-1/2}, r_{k+1/2}]$ which for the sake of simplicity we consider of uniform size, i.e for which $x_{i+1/2} - x_{i-1/2} = \Delta x \;  \forall i$, $y_{j+1/2} - y_{j-1/2} = \Delta x \; \forall j$ and $r_{k+1/2} - r_{k-1/2} = \Delta r \; \forall k$. 

{ Let $u$ be a solution of the mesoscopic equation \eqref{scaledmac} or the macroscopic equation \eqref{mactot}. Because in the microscopic model the parameter $\alpha$ is chosen small, the non-local term of the mesoscopic model can be considered local for the numerical simulations. Thus, both equation can be rewritten under the form:
$$\begin{array}{ll} \partial_t u + \partial_r (g(r) u) - &\nabla_x\cdot \Big( u(t,r,x) \nabla_x \,\xi \Big) - {\Dif} \Delta_x u \\
 &= 2^{1+\f{1}{d}}\beta(2^{\f{1}{d}}r) u(t, 2^{\f{1}{d}}r,x) - \beta(r) u(t,r,x). \end{array} $$
 We define
$$
\bar{u}_{i,j,k}(t) \approx \frac{1}{\Delta x \Delta y \Delta r} \iiint_{C_{i,j,k}} u(t,r,x) dx dy dr 
$$
as the cell averages of the calculated solution. A general semi-discrete finite-volume scheme for these equations can be written in the form}
\begin{align}
    \frac{\diff \bar{u}_{i,j,k}}{\diff t} = & { -}\frac{F^x_{i+1/2,j,k} - F^x_{i-1/2,j,k}}{\Delta x} { -} \frac{F^y_{i,j+1/2,k} - F^y_{i,j-1/2,k}}{\Delta y} \nonumber\\
    &  - \frac{G_{i,j,k+1/2} - G_{i,j,k-1/2}}{\Delta r} + 2^{1+\f{1}{d}} \beta(r_{\tilde{k}}) \bar{u}_{i,j,{\tilde{k}}} - \beta(r_k) \bar{u}_{i,j,k}. \label{scheme_macro}
\end{align}
 and $F^x_{i+1/2,j,k}, F^y_{i,j+1/2,k}, G_{i,j,k+1/2}$ the upwind numerical fluxes at the interfaces in the x-, y- and r- direction which approximate the continuous fluxes $-u(t,r,x) \partial_x \xi - \Dif \partial_x u$, $-u(t,r,x) \partial_y \xi - D \partial_y u$ and $g(r) u(t,r,x)$, respectively. To obtain formulae for numerical fluxes, we first construct piecewise constant functions in each cell $C_{i,j,k}$: 
$$
\tilde{u} (r,x) = \bar{u}_{i,j,k}, \; (r,x) \in C_{i,j,k}.
$$
Equipped with this reconstructed $\tilde{u}_{i,j,k}(r,x)$, the upwind fluxes are computed as
\begin{align*}
    F^x_{i+1/2,j,k} &= u_{i,j,k} \max(0,v^x_{i+1/2,j,k}) + u_{i+1,j,k} \min(0,v^x_{i+1/2,j,k}) - \Dif \frac{u_{i+1,j,k}-u_{i,j,k}}{\Delta x}\\
    F^y_{i,j+1/2,k} &= u_{i,j,k} \max(0,v^y_{i,j+1/2,k}) + u_{i,j+1,k} \min(0,v^y_{i,j+1/2,k}) - \Dif \frac{u_{i,j+1,k}-u_{i,j,k}}{\Delta y}\\
    G_{i,j,k+1/2} &= g u_{i,j,k}, 
\end{align*}
where we have used the fact that the growth term is a positive constant $g \geq 0$, and where
\begin{align*}
v^x_{i+1/2,j,k} = { -}\frac{\xi_{i+1,j,k} - \xi_{i,j,k}}{\Delta x}\\
v^y_{i,j+1/2,k} = { -}\frac{\xi_{i,j+1,k} - \xi_{i,j,k}}{\Delta y}.
\end{align*}
{ In the macroscopic model,} $\xi_{i,j,k}$ is a numerical approximation of $\int_{r_{\min}}^{r_{\max}} { \Gamma}(r_k,r') \tilde{u}(t,r',x_i,y_j,s) \diff r'$ at the point $(x_i,y_j,r_k)$ given by 
$$
\xi_{i,j,k} = \Delta r \sum_{\ell} { \Gamma}(r_k,r_{\ell}) \bar{u}_{i,j,{\ell}},
$$
{ and in the mesoscopic model $\xi_{i,j,k}$ is a numerical approximations of $\int_{r_{\min}}^{r_{\max}} (K^{\eps}(r_k,r',\cdot,\cdot) *_x\tilde{u}(t,r',\cdot,\cdot))(x_i,y_j) dr'$ at point $(x_i,y_j,r_k)$ given by
$$
\xi_{i,j,k} = \Delta r \sum_{\ell} \sum_{p,q} K(r_k,r_{\ell},x_p,y_q) \bar{u}_{i-p+1,j-q+1,{\ell}}.
$$
}
As mentionned earlier, we consider no-flux boundary conditions by setting $F^x_{1/2,j,k} = F^x_{N_x+1/2,j,k} = F^y_{i,1/2,k} = F^y_{i,N_x+1/2,k} = 0$, $\forall i,j,k$. As for the radii, we consider zero-flux boundary conditions at $r_{\min}$ and $r_{\max}$ to ensure mass conservation. This is accounted for by taking the fluxes at the interfaces in the $r$-direction such that 
$$
G_{i,j,N_r+1/2} = G_{i,j,1/2} = 0.
$$
\textit{Treatment of the fragmentation term}. The second fragmentation term corresponds to the approximation of $\frac{1}{\Delta x \Delta y \Delta r} \iiint_{C_{i,j,k}} \beta(r) \tilde{u}(t,r,x) dxdr$, where we have approximated $\int_{r_{k-1/2}}^{r_{k+1/2}} \beta(r) dr \approx \Delta r \beta(r_k)$. The integral $\frac{1}{\Delta x \Delta y \Delta r} \iiint_{C_{i,j,k}} \beta(\sqrt{2} r) \tilde{u}(t,\sqrt{2}r,x) dxdr$ from Equation~\eqref{mactot} requires a bit more attention. By performing a change of variable $s = \sqrt{2} r$, we have:
\begin{align*}
    \frac{1}{\Delta x \Delta y \Delta r} \iiint_{C_{i,j,k}}& \beta(\sqrt{2} r) \tilde{u}(t,\sqrt{2}r,x) dxdydr \\
    &= \frac{1}{\sqrt{2} \Delta x \Delta y \Delta r} \int_{x_{i-1/2}}^{x_{i+1/2}} \int_{y_{j-1/2}}^{y_{j+1/2}} \int_{\sqrt{2} r_{k-1/2}}^{\sqrt{2} r_{k+1/2}} \beta(r) \tilde{u}(t,r,x) dxdydr\\
    & \approx \beta(r_{\tilde{k}}) \bar{u}_{i,j,\tilde{k}},
\end{align*}
where $\tilde{k}$ is the index of the grid in which $\sqrt{2} r_k$ lives, i.e for which $\sqrt{2} r_{k} \in [r_{\tilde{k}-1/2} r_{\tilde{k}+1/2}]$. Notice that more sophisticated methods could be envisionned to account for the different numerical boxes that contain the boundaries $\sqrt{2} r_{k-1/2}$ and $\sqrt{2} r_{k+1/2}$ of the integral, but with our choice of parameters, we find that this approximation is sufficient for our purpose. 

\textit{Initial condition}
For {both the mesoscopic and} the macroscopic model, we  consider the continuum equivalent of the initial condition for the microscopic simulations, i.e a uniform distribution in size and radius. To this aim, we  consider the domain $[x_{min},x_{max}] \times [x_{min},x_{max}] \times [r_{\min}, r_{\max}]$ and set 
$$
u_0(0,x,r) = \frac{1}{\pi S^2(r_{\max}-r_{\min})} \1_{B(0,S)}(x), 
$$
Ensuring that the initial mass is $\iiint_{\Omega \times [r_{\min}, r_{\max}]} u_0(0,x,r) dx dy dr = 1$ so that $u_0$ corresponds to the density distribution of the individual particles initially considered for the microscopic simulations.

Finally, the semi-discrete scheme \eqref{scheme_macro} is discretized in time by the forward Euler method with non-uniform time steps $\Delta t^n$ to account for the CFL condition. The adaptative time-step is chosen such that:
$$
\Delta t^n = \min\bigg(\frac{\delta \min(\Delta x, \Delta y) }{\lVert F^x[u^n] \rVert_{\infty},\lVert F^y[u^n] \rVert_{\infty}},\frac{\Delta x \Delta y}{2 D}, \frac{\Delta r}{\lVert \beta \rVert_{L^1}}, \frac{\Delta r}{2g} \bigg).
$$
The parameters used for simulations of the macroscopic model are be the same as the ones used for the microcopic simulations, summarized in Table~\ref{tab:micro_param}. 

\subsection{Results}
In this section, we numerically explore the three (microscopic, mesoscopic and macroscopic) models in different settings. In Section~\ref{numerics_case1}, we consider the case where particles grow in size but do not divide, in Section~\ref{numerics_case2} we activate the cell division process but without particle growth, and finally Section~\ref{numerics_case3} presents the results of the general model with all ingredients. For comparing the different models, we look at different observables:

\textit{Total density.}
For the macroscopic model, we consider the density $\big(\bar{u}_{i,j,k} \big)_{1\leq i \leq N_x, 1 \leq j \leq N_y, 1\leq k \leq N_r}$ of the  model computed on the numerical grid points, and reconstruct the one from the microscopic simulation according to:
$$\bar{u}^N_{i,j,k} = \frac{1}{N \Delta x \Delta y \Delta r} \#_\ell \{ (i-1) \Delta x \leq X_\ell \leq i \Delta x, (j-1) \Delta y \leq Y_\ell \leq j \Delta y, (k-1) \Delta r \leq R_\ell \leq k \Delta r \}. $$

\textit{Spatial distribution.}
We denote  $u^{\text{spatial}}$ the spatial distribution, obtained as the integral in radii of the total density:
$$
u^{\text{spatial}}(x,t) = \int \tilde{u}(t,r,x) dr,
$$
computed in the microscopic setting to 
$$\bar{u}^{N,\text{spatial}}_{i,j} = \frac{1}{N \Delta x \Delta y} \#_\ell \{ (i-1) \Delta x \leq X_\ell \leq i \Delta x, (j-1) \Delta y \leq Y_\ell \leq j \Delta y \}. $$

\textit{Size distribution.}

We  denote $u^{\text{size}}$  the size distribution, obtained as the integral in space of the total density:
$$
u^{\text{size}}(r,t) = \iint \tilde{u}(t,r,x) dxdy,
$$
computed in the microscopic setting to 
$$\bar{u}^{N,\text{size}}_{k} = \frac{1}{N \Delta r} \#_\ell \{ (k-1) \Delta r \leq R_\ell \leq k \Delta r \}. $$

\textit{Radial density.}
As the particles are spreading radially, an efficient way to characterise the
dynamics is to compute a radial distribution $\big(u^{N,\text{radial}}_{i,k} \big)_{1\leq i \leq N_x, 1 \leq k \leq N_r}$, which gives the average number of cells in rings of size $\Delta x$:
$$
u^{N,\text{radial}}_{i,k} = \frac{1}{N (2i-1) \pi \Delta x^2} \#_\ell \{ (i-1)\Delta x \leq \sqrt{X_\ell^2+Y_\ell^2} \leq i \Delta x, (k-1) \Delta r \leq R_\ell \leq k \Delta r\}.
$$
In the macroscopic setting, this quantity will be referred to as $u^{\text{radial}}_{i,k} = \iint_{(i-1)\Delta x \leq |x| \leq i \Delta x} \tilde{u}(t,r_k,x) dx$.

\subsubsection{Case 1: No particle fragmentation, constant growth}\label{numerics_case1}
{ In this section, we deactivate particle fragmentation and consider a constant growth rate with $g=0.008$. The other parameters are given in Table~\ref{tab:micro_param}. In order to account for the stochasticity of the particle system, we perform 6 simulations for each parameter set and average the densities. 

In Fig.\ref{figgrowthnofrag}, we show the solutions of the three models at times $t=4$ (panel A) and $t=60$ (panel B). In each panel, the first column shows the spatial density $u^{spatial}(x,y,t)$ computed on the three models and plotted as function of the space variables $x$ (abscissa) and $y$ (ordinates). The second column shows the radial density $u^{radial}(\lambda,r,t)$ as function of the radial variable $\lambda$ (abscissa) and size variable $r$ (ordinates). In each column, the first line shows the solution of the microscopic model for $N = 100$ particles, the second line for $N_0=2000$ particles, the third line the solution of the mesoscopic model and the fourth line the solution of the macroscopic model. The figure on the top right represents the size density $u^{size}(r,t)$ as function of the particle size variable $r$, while the bottom right panel shows the radial distribution $u^{radial}(\lambda, t)$ as function of the radial variable $\lambda$. Blue curves are the macroscopic quantities, orange curves the mesoscopic quantities, and the dotted lines the microscopic quantities (yellow for $N_0=100$, green for $N_0 = 500$,  light blue for $N_0=2000$ and red for $N_0 = 4000$). }

\begin{figure}[H]
\centering
   \includegraphics[scale = 0.67]{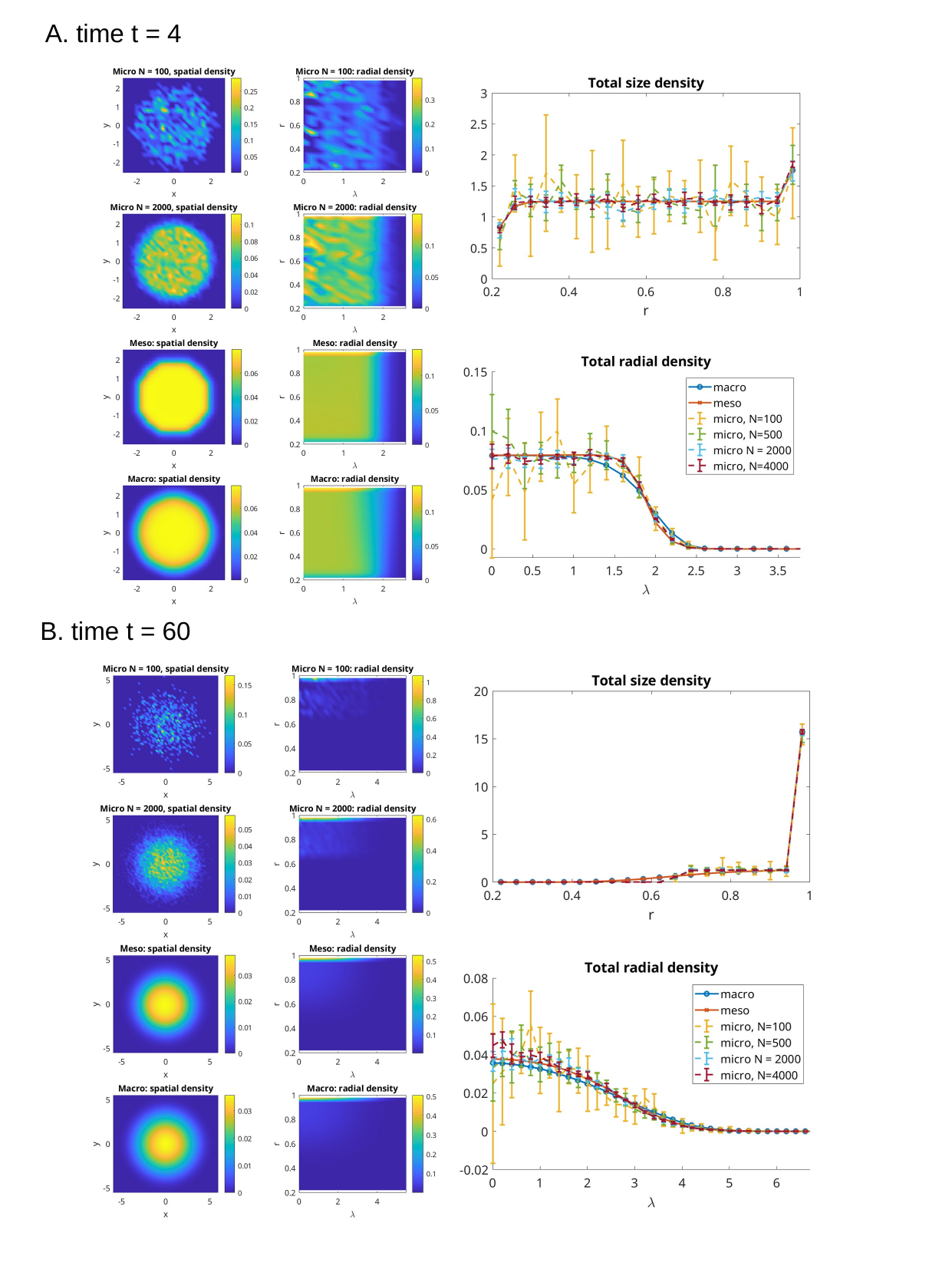}
\vspace{-1cm}
   \caption{Numerical simulations without particle fragmentation and with constant growth, at times $t=4$ (panel A) and $t=60$ (panel B). In each panel, the first column shows the spatial density $u^{spatial}(x,y,t)$ plotted as function of the space variables $x$ (abscissa) and $y$ (ordinates). The second column shows the radial density $u^{radial}(\lambda,r,t)$ as function of the radial variable $\lambda$ (abscissa) and size variable $r$ (ordinates). In each column, the first line shows the solution of the microscopic model for $N = 100$ particles, the second line for $N_0=2000$ particles, the third line the solution of the mesoscopic model and the fourth line the solution of the macroscopic model. The figure on the top right represents the size density $u^{size}(r,t)$ as function of the particle size variable $r$, while the bottom right panel shows the radial distribution $u^{radial}(\lambda, t)$ as function of the radial variable $\lambda$. Blue curves are the macroscopic quantities, orange curves the mesoscopic quantities, and the dotted lines the micro quantities (yellow for $N_0=100$, green for $N_0 = 500$,  light blue for $N_0=2000$ and red for $N_0 = 4000$). \label{figgrowthnofrag}}
\end{figure}

{As one can observe in Fig. \ref{figgrowthnofrag}, we obtain a very good agreement between the simulations of the three models both for the size distributions (top right panel) and for the radial distributions (bottom right panel). Although there is a large variability in the simulations for $N = 100$ particles (yellow errorbar curves), the average density is already quite close to the macroscopic quantities, and the agreement is even better for $N_0=2000$ particles (light blue curves). As expected, the particle radii concentrate at $r_{\max}=1$ because of the growth term, (second column and top right figure in panel B) and we observe a radial spatial diffusion (first columns in panels A and B).} 

{In order to quantify the errors between the models, we plot in Fig. \ref{fig_growthnofragERRORS} the relative $L^1$ errors between the three quantities of interest. 

Given a density $u_b(t,r,x)$, we compute three relative errors between $u_b$ compared with the mesoscopic density $u_{mes}$ defined as:

\begin{align*}
    E_{tot} &= \frac{\lVert u_{mes} - u_{b}\rVert_{L^1(\Omega \times [r_{\min}, r_{\max}])}}{\lVert u_{mes} \rVert_{L^1(\Omega \times [r_{\min}, r_{\max}])}}\\
    E_{spatial} &= \frac{\lVert u^{\text{spatial}}_{mes} - u^{\text{spatial}}_b\rVert_{L^1(\Omega)}}{\lVert u^{\text{spatial}}_{mes}\rVert_{L^1(\Omega)}}\\
        E_{size} &= \frac{\lVert u^{\text{size}} - u^{\text{size}}_b\rVert_{L^1([r_{\min}, r_{\max}])}}{\lVert u^{\text{size}}_{mes}\rVert_{L^1([r_{\min}, r_{\max}])}}
\end{align*}

In Fig. \ref{fig_growthnofragERRORS}, dotted lines are obtained when $u_b$ is the microscopic density for $N_0=100$ (blue curves), $N_0=500$ (orange curves) and $N_0=2000$ (yellow curves), $ N_0=4000$ (purple curves) while the continuous black line are obtained when $u_b$ is the density of the macroscopic model. Note that in any case, we use the density of the mesoscopic model as reference for computing the relative errors. }

\begin{figure}[H]
    \centering
   \includegraphics[scale = 0.8]{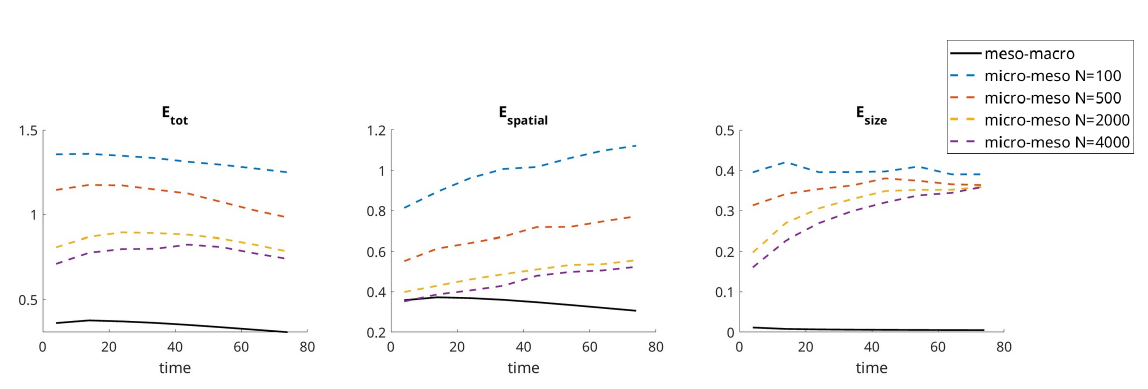}
    \caption{ $L^1$ relative errors between the three quantities of interest as function of time without particle fragmentation and with constant growth: $E_{tot}$ (left figure), $E_{spatial}$ (middle figure) and $E_{size}$ (right figure). Dotted lines are the relative errors between the meso and micro models for $N_0=100$ (blue curves), $N_0=500$ (yellow curves), $N_0=2000$ (yellow curves) and $N_0=4000$ (purple curves). Black continuous lines are the relative errors between the meso and macro models. }
    \label{fig_growthnofragERRORS}
\end{figure}
{As one observes in Fig. \ref{fig_growthnofragERRORS}, all the $L^1$ relative errors between the micro and meso models decrease as the number of particles $N_0$ increases (compare blue, red, yellow and purple dotted lines respectively). These results show that the microscopic model gets closer to the mesoscopic model as $N$ increases. It is noteworthy that the relative errors between the microscopic and mesoscopic models increase in time both when comparing the spatial and the size distributions (dotted lines in the middle and right panels of Fig. \ref{fig_growthnofragERRORS}). This may be due to the fact that microscopic simulations are made with a finite number of particles while the mesoscopic model is obtained in the limit of an infinite number of particles. As time goes, the average distance between particles diminishes because of repulsion, leading to less effective repulsive interactions in the microscopic setting compared to the mesoscopic one. These observations are supported by the fact that the discrepancy between the micro and meso models is slower when increasing the number of particles.     

Moreover, we observe that the relative errors between the macroscopic and the mesoscopic models are very small (black curves). This shows that the macroscopic model is a good approximation of the mesoscopic model in this regime of parameters. The fact that the two models are very close already for $\varepsilon = 1$ (where the meso model features non-local interactions while the macroscopic model is in the limit of local interactions) hints to the fact that linear diffusion dominates the non-local effects due to the repulsive interactions (non-linear diffusion term in the macro model). These observations are supported by the fact that the relative error between the two models decreases in time. Indeed, as time goes particles get farther from each other because of repulsion, decreasing the repulsive interactions for the benefit of linear diffusion. Therefore, it is expected that in time non-local effects vanish and we get a better agreement between the mesoscopic and macroscopic dynamics. We send the interested reader to Appendix~\ref{sec_nogrowthnofrag}, where we illustrate the role of non-local interactions by using a smaller linear diffusion coefficient.
}

\subsubsection{Case 2: No growth, particle division}\label{numerics_case2}
{In this section, we explore the case where particles undergo cell division but no growth. Because of the choice of the fragmentation kernel $\beta$, we expect all particles to end with the minimal radius $r_{\min}$. We adopt the same visualization as in previous section and show in Fig. \ref{fignogrowthfrag} the simulation results for $\Dif=0.01$, no growth and cell division, and in Fig. \ref{fig_nogrowthfragERRORS} the evolution in time of the relative errors between the three models. 
}

\begin{figure}[H]
  \centering \includegraphics[scale = 0.67]{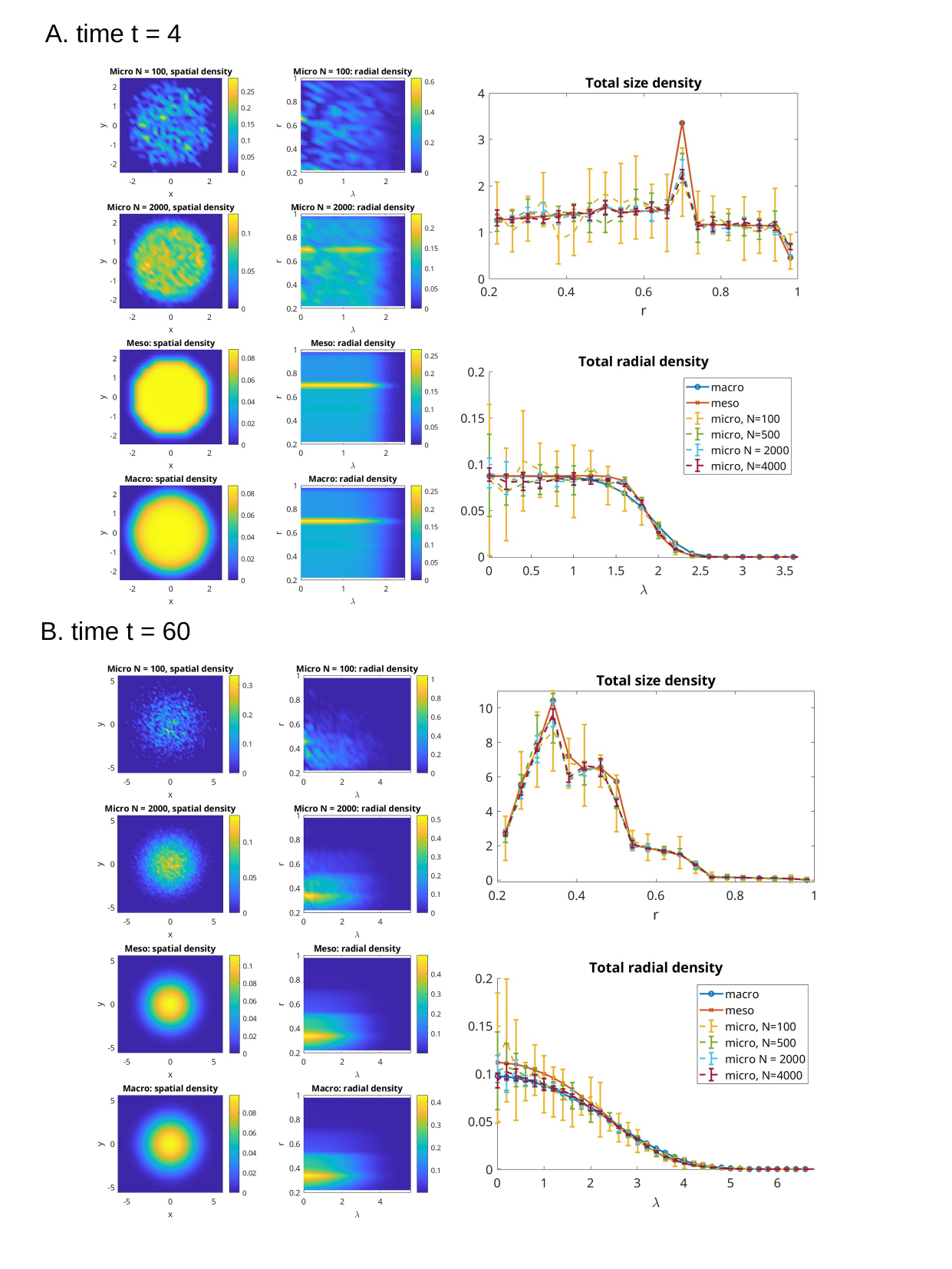}
    \vspace{-1cm}
    
    \caption{Numerical simulations with particle fragmentation and no growth, at times $t=4$ (panel A) and $t=60$ (panel B). In each panel, the first column shows the spatial density $u^{spatial}(x,y,t)$ plotted as function of the space variables $x$ (abscissa) and $y$ (ordinates). The second column shows the radial density $u^{radial}(\lambda,r,t)$ as function of the radial variable $\lambda$ (abscissa) and size variable $r$ (ordinates). In each column, the first line shows the solution of the microscopic model for $N = 100$ particles, the second line for $N_0=2000$ particles, the third line the solution of the mesoscopic model and the fourth line the solution of the macroscopic model. The figure on the top right represents the size density $u^{size}(r,t)$ as function of the particle size variable $r$, while the bottom right panel shows the radial distribution $u^{radial}(\lambda, t)$ as function of the radial variable $\lambda$. Blue curves are the macroscopic quantities, orange curves the mesoscopic quantities, and the dotted lines the micro quantities (yellow for $N_0=100$, green for $N_0 = 500$,  light blue for $N_0=2000$ and red for $N_0 = 4000$). \label{fignogrowthfrag}}.
\end{figure}

\begin{figure}[H]
    \centering
   \includegraphics[scale = 0.8]{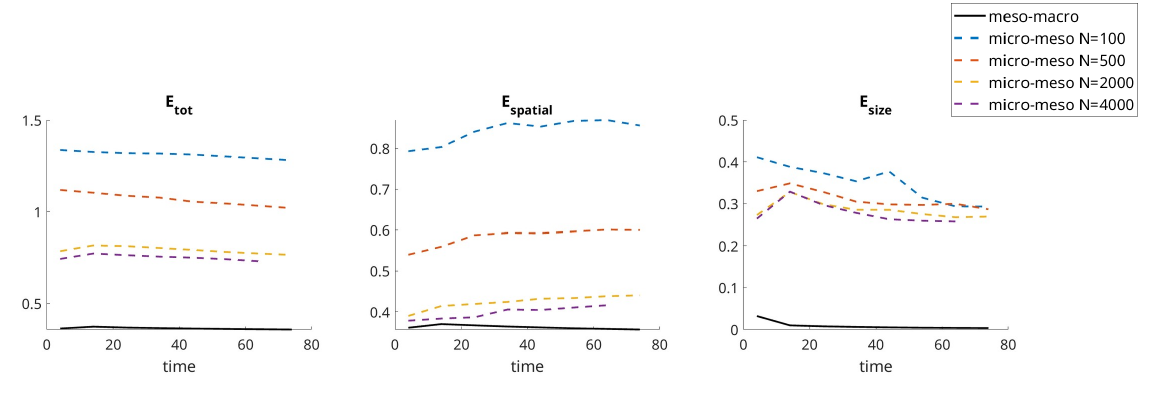}
    \caption{$L^1$ relative errors between the three quantities of interest as function of time with particle fragmentation and without growth: $E_{tot}$ (left figure), $E_{spatial}$ (middle figure) and $E_{size}$ (right figure). Dotted lines are the relative errors between the meso and micro models for $N_0=100$ (blue curves), $N_0=500$ (yellow curves), $N_0=2000$ (yellow curves) and $N_0=4000$ (purple curves). Black continuous lines are the relative errors between the meso and macro models. }
    \label{fig_nogrowthfragERRORS}
\end{figure}

{As one can observe, we again obtain a very good agreement between the micro, meso and macro models, in the temporal evolution of the radius as well as of  the spatial distributions. The distribution in radii (top right panels of Fig. \ref{fignogrowthfrag}) evolves in steps, concentrating successively to the attractive points $\frac{r_{\max}}{\sqrt{2}}$, $\frac{r_{\max}}{2}$, $\frac{r_{\max}}{2\sqrt{2}}$, $\frac{r_{\max}}{4}$, etc. From Fig. \ref{fig_nogrowthfragERRORS}, we observe that the $L^1$ relative error remains constant in time, which suggests that the fragmentation process leads to a longer agreement between the micro and meso models compared to the growth process alone (see previous section). These results can be due to the fact that the fragmentation process creates particles, therefore keeps the number of interactions high, enabling the micro system to spread efficiently. In all cases again, we note that the error decreases as the number of particles $N$ increases, suggesting that the mesoscopic model is a good approximation of the particle dynamics as $N$ increases.
}

\subsubsection{Case 3: Growth, particle division}\label{numerics_case3}
{Finally, we  look at the simulations activating both the growth and fragmentation. As previously, we show in Figs. \ref{figgrowthfrag}-\ref{figgrowthfrag2} the simulation results for $D=0.01$ with growth and cell division for $t=4$, $t=20$ (Fig. \ref{figgrowthfrag}) and $t=60$ (Fig. \ref{figgrowthfrag2}). }

\begin{figure}[H]
   \includegraphics[scale = 0.67]{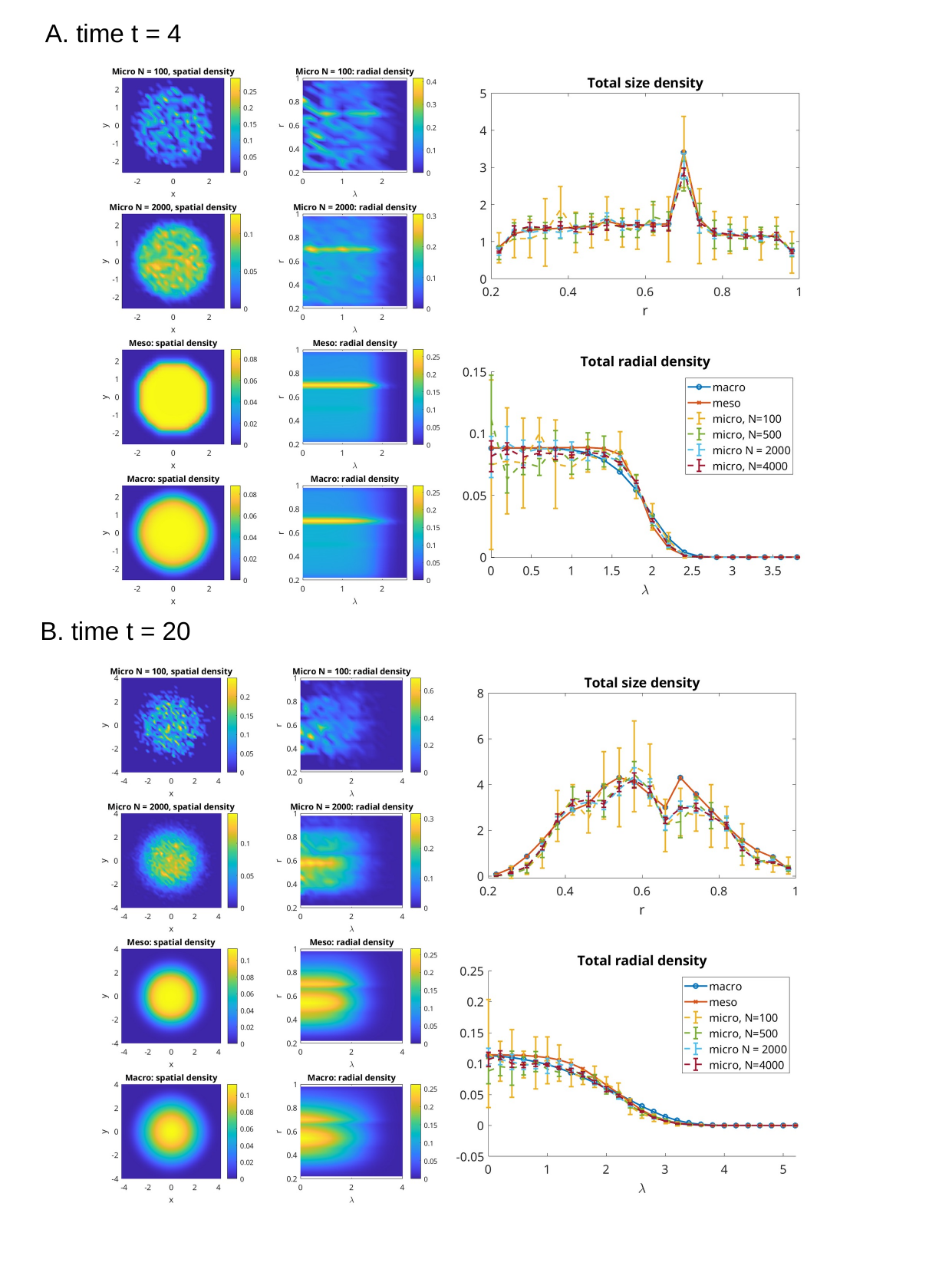}
    \vspace{-1cm}
    
    \caption{Numerical simulations with particle fragmentation and growth, at times $t=4$ (panel A) and $t=20$ (panel B). In each panel, the first column shows the spatial density $u^{spatial}(x,y,t)$ plotted as function of the space variables $x$ (abscissa) and $y$ (ordinates). The second column shows the radial density $u^{radial}(\lambda,r,t)$ as function of the radial variable $\lambda$ (abscissa) and size variable $r$ (ordinates). In each column, the first line shows the solution of the microscopic model for $N = 100$ particles, the second line for $N_0=2000$ particles, the third line the solution of the mesoscopic model and the fourth line the solution of the macroscopic model. The figure on the top right represents the size density $u^{size}(r,t)$ as function of the particle size variable $r$, while the bottom right panel shows the radial distribution $u^{radial}(\lambda, t)$ as function of the radial variable $\lambda$. Blue curves are the macroscopic quantities, orange curves the mesoscopic quantities, and the dotted lines the micro quantities (yellow for $N_0=100$, green for $N_0 = 500$,  light blue for $N_0=2000$ and red for $N_0 = 4000$). \label{figgrowthfrag}}.
\end{figure}

\begin{figure}[H]
   \includegraphics[scale = 0.7]{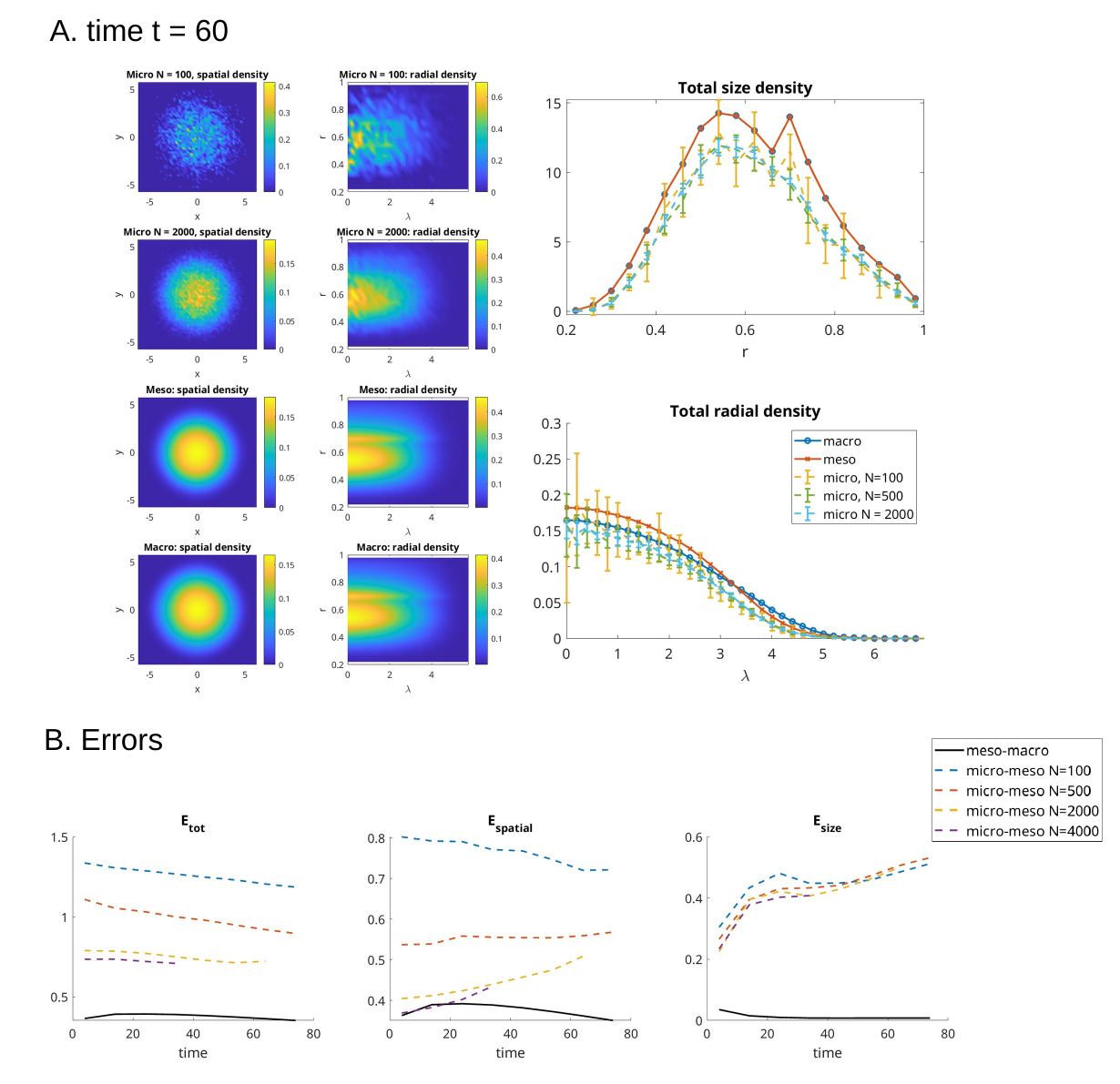}
    \vspace{-0.5cm}
    
    \caption{Panel A: Numerical simulations with particle fragmentation and growth at time $t=60$. Panel B: $L^1$ relative errors between the three quantities of interest as function of time: $E_{tot}$ (left figure), $E_{spatial}$ (middle figure) and $E_{size}$ (right figure). Dotted lines are the relative errors between the meso and micro models for $N_0=100$ (blue curves), $N_0=500$ (yellow curves), $N_0=2000$ (yellow curves) and $N_0=4000$ (purple curves). Black continuous lines are the relative errors between the meso and macro models.  \label{figgrowthfrag2}}
\end{figure}

 {Note that with both growth and fragmentation, the number of particles increases exponentially and so does the computation time. Because of computational efficiency, we stop the simulations when the number of particles exceeds $10000$ in the microscopic setting, corresponding to time $t\approx 64$ for $N_0 = 2000$ and $t\approx 36$ for $N_0 = 4000$. 
From Figs. \ref{figgrowthfrag}, we observe a good agreement between both models at early times of the simulations. The radii distribution reaches the expected profile at $t=20$, but after some time the mesoscopic and macroscopic models produce mass faster than the microscopic dynamics (top right figure of Fig. \ref{figgrowthfrag2} panel A and right figure in panel B). Again, these results are reminiscent of our previous observations. The mesoscopic and macroscopic models are obtained in the limit of infinite number of particles while the microscopic simulations are done with a finite number of particles. The errors in the initial condition are amplified by the growth fragmentation process. 

Altogether, these numerical results suggest that the micro and macro models are in very good agreement at least at early times of the growth fragmentation process, and that the agreement gets better as the number of particles in the microscopic setting $N$ increases. This suggest that the macroscopic model is a good approximation of the underlying macroscopic dynamics, that enables to overcome the problem of large computational cost raised by the microscopic model.    
}

\section{Conclusion and perspectives}
{ In this article, following the biological motivation of a more realistic mechanical model~\cite{doumic2020purely} on the modelling side and the theoretical study of the localisation limit for a discrete multispecies model on the mathematical side~\cite{doumic2024multispecies}, we proposed  a size and space stochastic individual-based model. We studied its asymptotics in two successive regimes: a mean-field limit, when the number of cells tends to infinity, then a localisation limit, when the size of the interaction domain between cells vanishes. 
For this last limit, we split the difficulties and proved a convergence result in the case without growth and division: the study of the full non-conservative equation is left for future work. Due to the lack of compactness of the equation in the size variable -- at least in the absence of growth -- we also required a strictly positive diffusion term and some regularity of the interaction potential. We have also explored the connections between the models numerically in the cases for which we do not yet have  a theoretical result, including growth and division, and also with a vanishing  diffusion coefficient.  
At first sight, the macroscopic model may appear more relevant for biology than the mesoscopic limit, since in most real life applications living cells only interact with a small number of neighbours. However, being derived from the mesoscopic limit, the model contains the fact that there is interaction with an infinite number of neighbours -- though these interactions tend to localise. Up to our knowledge, the direct derivation of an adequate macroscopic model from a microscopic one remains an open problem.}

{\bf Acknowledgments.} The authors warmly thank Nicolas Fournier and Benoit Perthame for illuminating discussions. 

\bibliographystyle{plain}       
\bibliography{biblioDHHP}  

\appendix
\section{Numerical simulations}
\label{sec_nogrowthnofrag}
\subsection{No growth, no fragmentation}

{In this section, we focus on the role of spatial interactions, by desactivating the growth and fragmentation processes. In Fig. \ref{fignogrowthnofrag}, we adot the same visualization as in the numerical section of the main paper, and we show the results of the micro meso and macro models with diffusion coefficient $\Dif=0.01$ at time $t=60$ (panel A) and the relative $L^1$-errors using the mesoscopic model as reference as previously (panel B). 
\begin{figure}[H]
   \includegraphics[scale = 0.7]{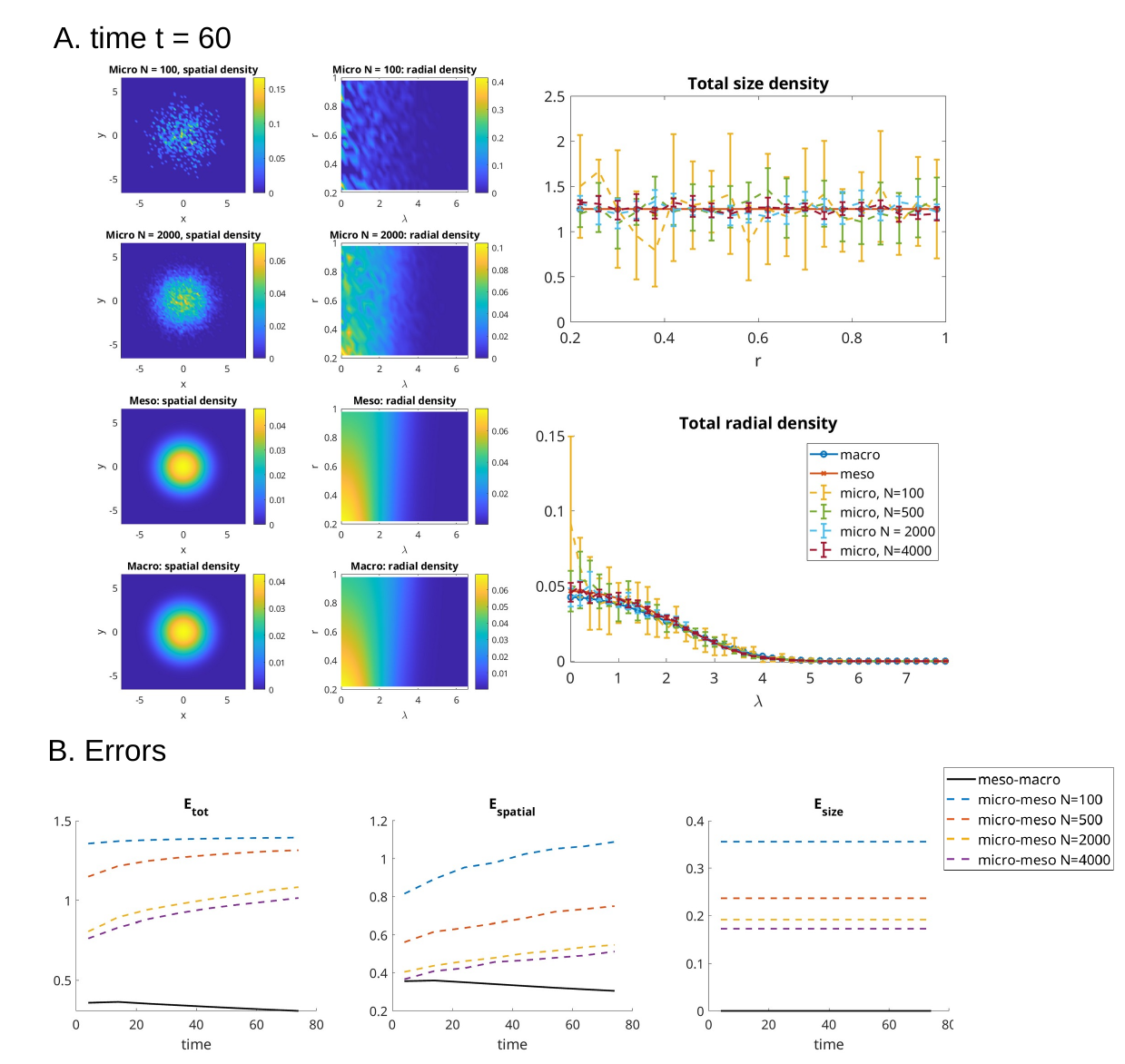}
    \vspace{-0.1cm}
    
    \caption{Panel A: Numerical simulations without particle fragmentation and without growth at time $t=60$. Panel B: $L^1$ relative errors between the three quantities of interest as function of time: $E_{tot}$ (left figure), $E_{spatial}$ (middle figure) and $E_{size}$ (right figure). Dotted lines are the relative errors between the meso and micro models for $N_0=100$ (blue curves), $N_0=500$ (yellow curves) and $N_0=2000$ (yellow curves). Black continuous lines are the relative errors between the meso and macro models.  \label{fignogrowthnofrag}}
\end{figure}

Fig. \ref{fignogrowthnofrag} panel A shows that the spatial distribution of particles spreads radially from the center (left column of panel A), as expected since the interaction kernel is spatially isotropic. Interestingly for a fixed population of cells, we observe a size-dependent spatial spreading (middle column of panel A). Indeed, small cells seem to concentrate in the middle of the domain while larger cells seem to diffuse farer from the center. This is due to our choice for the repulsion kernel $K$. Indeed, by choosing $\gamma(r) = r$ in \eqref{choiceK}, the repulsion interaction intensity between particles of sizes $r$ and $s$ is $\alpha(r,s) = rs$, meaning that pairs of large particles will interact stronger than pairs of small particles, leading to larger diffusion. Note that because the interaction kernel $K$ is of the form of spatial Gaussian functions of variance $(r+s)^2$, it also involves that larger particles will interact farer than small particles, and therefore with more particles in their vicinity. This can also lead to a larger spreading of large particles compared to small particles. Note however that this last effect can only be visible in the micro and meso setting, since they account for the non-locality of the interaction. Indeed in the localisation limit (i.e for the macro model), the Gaussian functions integrate to 1 and as a result, the interaction strength between particles of sizes $r$ and $s$ is only controlled by $\gamma(r) \gamma(s)$.  

Coming back to Fig. \ref{fignogrowthfrag} panel B, we observe that the micro, meso and macro models are in good agreement, and that the relative $L^1$-error between the micro and the meso model decreases as $N$ increases (compare the dotted lines of Fig. \ref{fignogrowthnofrag} panel B). One observes again that the spatial relative error between both models increases in time, as for the case with growth only (section \ref{numerics_case1}). These results confirm that the discrepancies between the micro and meso models are mostly controled by spatial interactions. As time goes, particles get farer from each other in the microscopic setting, disminishing the number of interactions and leading to less agreement with the mesoscopic model featuring an infinite number of particles. 

\subsection{No growth, no fragmentation -- no diffusion}
Here, we go a step further on the analysis of the role of nonlinear interactions in the link between the three models. To this aim, we perform simulations without growth, without fragmentation, and set the linear diffusion coefficient $\Dif = 0$. We adopt the same representation as previously and show in Fig. \ref{fignogrowthnofragD0} panel A the solution of the three models and in panel B the relative $L^1$-errors using the mesoscopic model as reference. 

\begin{figure}[H]
   \includegraphics[scale = 0.7]{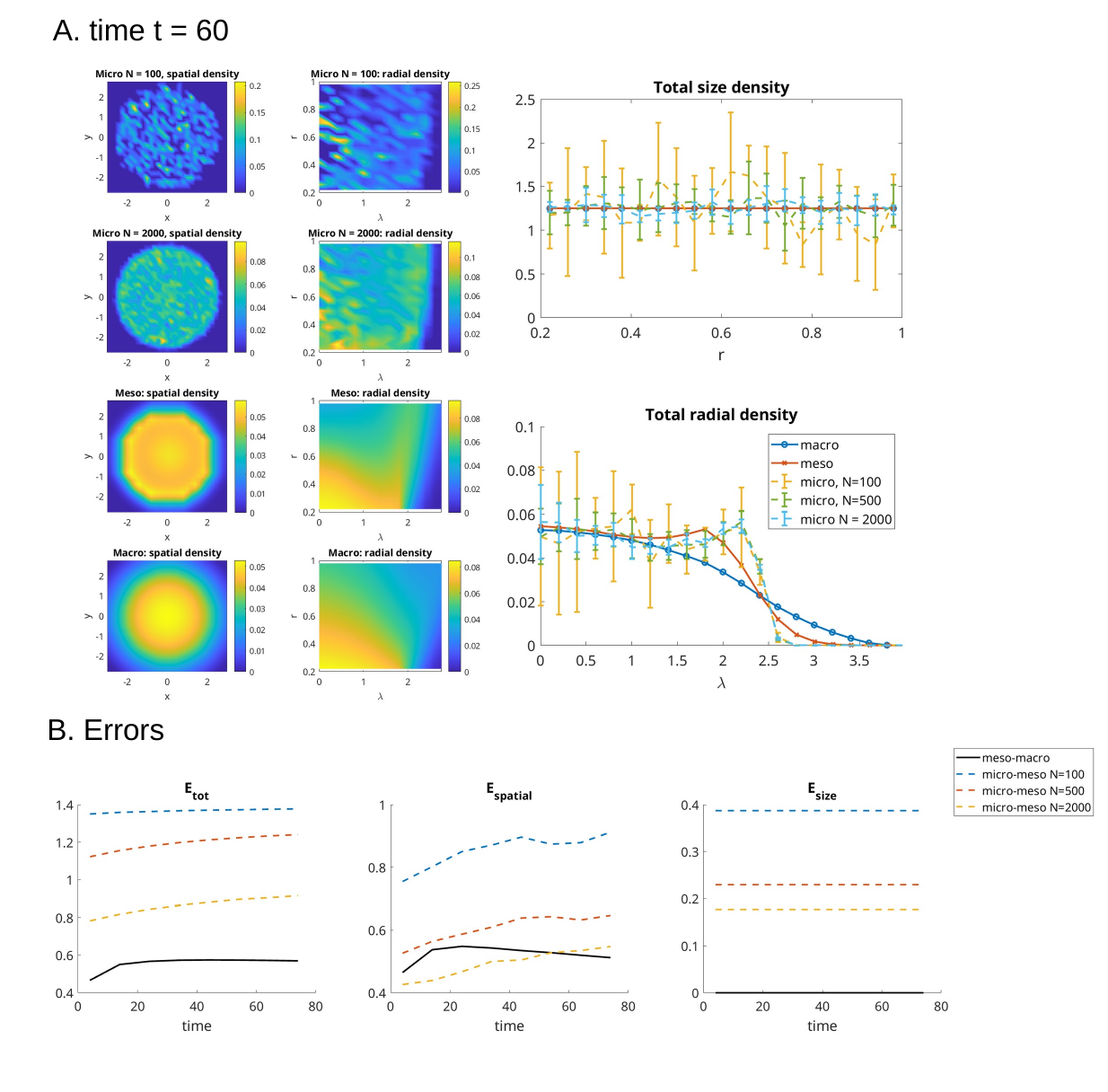}
    \vspace{-1cm}
    
    \caption{Panel A: Numerical simulations without particle fragmentation, without growth and for $\Dif = 0$, at time $t=60$. Panel B: $L^1$ relative errors between the three quantities of interest as function of time: $E_{tot}$ (left figure), $E_{spatial}$ (middle figure) and $E_{size}$ (right figure). Dotted lines are the relative errors between the meso and micro models for $N_0=100$ (blue curves), $N_0=500$ (yellow curves) and $N_0=2000$ (yellow curves). Black continuous lines are the relative errors between the meso and macro models.  \label{fignogrowthnofragD0}}
\end{figure}

As one can see in Fig. \ref{fignogrowthnofragD0}, the size sorting effect previously observed is even stronger when linear diffusion is deactivated (compare with Fig. \ref{fignogrowthnofrag}). These are expected results since linear diffusion tends to smoothen the solution. With nonlinear diffusion only, we observe a concentration of cells in rings located on the boundary of the solution spatial support (see bottom right figure of panel A) for the micro and meso models, while the radial density of the macro model remains monotically decreasing from the center. 

As one can observe in Fig. \ref{fignogrowthnofragD0} panel B, the relative error between the meso and macro models is larger with nonlinear diffusion only compared to the case with linear diffusion (compare continuous black lines of Fig. \ref{fignogrowthnofragD0} with Fig. \ref{fignogrowthnofrag} panel B). Altogether, these results illustrate the effects of nonlocal interactions and highlight the essential smoothening role of linear diffusion.

\end{document}